\documentclass[11pt,letterpaper]{amsart}
\usepackage[utf8]{inputenc}
\usepackage{enumerate}
\usepackage{amsmath}
\usepackage{cite}
\usepackage{amsfonts}
\usepackage{mathrsfs}
\usepackage{galois}
\usepackage[toc,page]{appendix}
\usepackage{xcolor}
\usepackage[colorlinks, linkcolor = blue, anchorcolor = blue, citecolor = blue]{hyperref}

\theoremstyle{definition}
\newtheorem{Def}{Definition}[section]
\newtheorem{Rem}[Def]{Remark}

\theoremstyle{plain}
\newtheorem{Thm}[Def]{Theorem}

\newtheorem{Lem}[Def]{Lemma}
\newtheorem{Prop}[Def]{Proposition}
\newtheorem{Cor}[Def]{Corollary}
\newtheorem{claim}{Claim}

\numberwithin{equation}{section}

\newcommand{\AAa}{\mathbb{A}}
\newcommand{\BB}{\mathbb{B}}

\newcommand{\NN}{\mathbb{N}}

\newcommand{\RR}{\mathbb{R}}
\newcommand{\SSp}{\mathbb{S}}

\newcommand{\ZZ}{\mathbb{Z}}

\newcommand{\cA}{\mathcal{A}}

\newcommand{\cC}{\mathcal{C}}

\newcommand{\cF}{\mathcal{F}}

\newcommand{\cH}{\mathcal{H}}

\newcommand{\cL}{\mathcal{L}}

\newcommand{\cP}{\mathcal{P}}

\newcommand{\cR}{\mathcal{R}}

\newcommand{\cV}{\mathcal{V}}
\newcommand{\cW}{\mathcal{W}}

\newcommand{\cZ}{\mathcal{Z}}

\newcommand{\scB}{\mathscr{B}}
\newcommand{\scC}{\mathscr{C}}

\newcommand{\scE}{\mathscr{E}}
\newcommand{\scF}{\mathscr{F}}
\newcommand{\scG}{\mathscr{G}}

\newcommand{\scM}{\mathscr{M}}

\newcommand{\scR}{\mathscr{R}}

\newcommand{\scU}{\mathscr{U}}

\newcommand{\mbfC}{\mathbf{C}}

\newcommand{\mbfF}{\mathbf{F}}

\newcommand{\mbfI}{\mathbf{I}}

\newcommand{\mbfM}{\mathbf{M}}

\newcommand{\SCap}{\mathbf{SCap}}

\newcommand{\IV}{\mathcal{IV}}

\newcommand{\set}[1]{\left\{#1\right\}}
\newcommand{\Sing}{\operatorname{Sing}}
\newcommand{\supp}{\operatorname{supp}}
\newcommand{\BV}{\operatorname{BV}}

\title{Generic regularity of minimal hypersurfaces in~dimension~$8$}
\author{Yangyang Li}
\address{Department of Mathematics, Princeton University, Fine Hall, 304 Washington Road, Princeton, NJ 08540, USA}
\email{yl15@math.princeton.edu}
\author{Zhihan Wang}
\address{Department of Mathematics, Princeton University, Fine Hall, 304 Washington Road, Princeton, NJ 08540, USA}
\email{zhihanw@math.princeton.edu}
\date{}
\thanks{Y. Li was partially supported by NSF-DMS-1811840}

\begin{document}
\bibliographystyle{abbrvalpha}

\begin{abstract}
    In this paper, we show that every $8$-dimensional closed Riemmanian manifold with $C^\infty$-generic metrics admits a smooth minimal hypersurface. This generalized previous results by N. Smale \cite{smaleGenericRegularityHomologically1993} and Chodosh-Liokumovich-Spolaor \cite{chodoshSingularBehaviorGeneric2020}. Different from their local perturbation techniques, our construction is based on a global perturbation argument in \cite{wangDeformationsSingularMinimal2020} and a novel geometric invariant which counts singular points with suitable weights.
\end{abstract}
\maketitle

\section{Introduction}

    The interior regularity theory of minimal hypersurfaces dates back to around 1970, when H. Federer \cite{federerSingularSetsArea1970} applied his dimension reduction arguments to show that the singular set of any area-minimizing rectifiable hypercurrent has Hausdorff codimension at least $7$ (away from its boundary). His proof relied on a previous work by J. Simons \cite{simonsMinimalVarietiesRiemannian1968} on the nonexistence of stable cones in $\mathbb{R}^{n+1}$ for $2 \leq n \leq 6$. In particular, in a Riemannian manifold $M^{n+1}$ with $2 \leq n \leq 6$, an area-minimizing integral $n$-cycle must be a smooth minimal hypersurface. However, even when $n = 7$, Bombieri-De Giorgi-Giusti \cite{bombieriMinimalConesBernstein1969} has proved that the famous Simons cone (introduced in \cite{simonsMinimalVarietiesRiemannian1968})
    \begin{equation}
        C = \set{(x, y) \in \mathbb{R}^4 \times \mathbb{R}^4 : |x| = |y|}\,,
    \end{equation}
    is an area-minimizing hypercone in $\mathbb{R}^8$, which has a singular point at the origin.

    The similar phenomenon occurs in the stable minimal hypersurface (varifold) setting with natural a priori assumption on the singular set, which was originally proved by Schoen-Simon \cite{schoenRegularityStableMinimal1981} and later improved by N. Wickramasekara \cite{wickramasekeraGeneralRegularityTheory2014}. This deep regularity theorem plays an important role in the Almgren-Pitts min-max theory \cite{almgrenTheoryVarifoldsVariational1965,pittsExistenceRegularityMinimal1981} extending the original regularity result obtained from Schoen-Simon-Yau \cite{schoenCurvatureEstimatesMinimal1975} to higher dimensional closed manifolds. Quantitative versions of the partial regularity theorem and rectifiability of singular set can be found in \cite{cheegerQuantitativeStratificationRegularity2013,naberSingularStructureRegularity2020}.

    In general, the existence of Simons cone has shown that even in dimension $8$, a minimal hypersurface can have singular set. In fact, N. Smale \cite{smaleMinimalHypersurfacesMany1989} has constructed explicit examples of minimal hypersurfaces with boundary containing arbitrarily many singular points virtually by applying bridge principle on stable cones. Later, in \cite{smaleSingularHomologicallyArea1999}, he also constructed a closed Riemannian manifold admitting a unique area-minimizer in a homology class, which contains $2$ singular points. \cite{Simon2021Sing} shown that any closed subset in $\RR^{n-7}$ can be realized as singular set of some stable minimal hypersurface in $(\RR^{n+1}, g)$ for some smooth metric $g$ arbitrarily close to the Euclidean one, making the rectifiability result of singular sets by \cite{naberSingularStructureRegularity2020} sharp in some way.

    This would naturally stimulate one to opine the opposite:\\

    \noindent\textbf{Q 1}: \textit{Is it possible to resolve singularities of a minimal hypersurface by a small perturbation of the ambient metric?}\\

    The same question on area-minimizing hypersurfaces has been raised in \cite{10.2307/j.ctt1bd6kkq.37,smaleMinimalHypersurfacesMany1989}. An affirmative answer to this question together with the construction of minimal hypersurfaces from Almgren-Pitts min-max theory would directly lead to the generic existence of smooth minimal hypersurfaces.\\

    \noindent\textbf{Q 2}: \textit{Does there exist a smooth minimal hypersurface in a closed Riemannian manifold with generic metrics?}\\

    In a closed manifold $M$ of dimension $8$, N. Smale \cite{smaleGenericRegularityHomologically1993} confirmed the existence of such a small perturbation for an area minimizer in its homology classes and thus, the generic existence of smooth minimal hypersurfaces. Recently, Chodosh-Liokumovich-Spolaor \cite{chodoshSingularBehaviorGeneric2020} justified the generic existence in the positive Ricci curvature case by constructing an optimal nested volume parametrized sweepout for $1$-parameter min-max minimal hypersurfaces. Both of their arguments essentially employed the local analysis of isolated singular points by Hardt-Simon \cite{hardtAreaMinimizingHypersurfaces1985}, which produces unique Hardt-Simon foliations of smooth minimal hypersurfaces in either side of a regular area-minimizing cone. Similar results were obtained for regular one-sided area-minimizing cone by F. Lin \cite{Lin_1987} and Z. Liu \cite{liuStationaryOnesidedAreaminimizing2019}. Nevertheless, by a calibration argument \cite{bombieriMinimalConesBernstein1969,lawlorSufficientCriterionCone2012}, the existence of such a smooth foliations implies that the regular minimal cone should be at least one-sided area-minimizing, and thus, it seems impossible to apply this local analysis directly on a regular stable minimal cone which is not area-minimizing one either side. To the best of the authors' knowledge, the existence of such a cone is still widely open, and a negative answer to this will significantly simplify our arguments and other important analysis in the literature.

    In this paper, we shall utilize a global perturbation argument from \cite{wangDeformationsSingularMinimal2020} (See Lemma \ref{Lem_Pre_Perturb to find MH reg}) and a novel geometric invariant, ``singular capacity'' (See Definition \ref{Def_Sing Cap}) to resolve singular points by small perturbations. As an application, we are able to settle the second question and prove the generic existence of a smooth minimal hypersurface in a given closed Riemannian manifold.

    \begin{Thm}[Main Theorem] \label{Thm_Intro_Main}
        Let $M$ be an $8$-dimensional closed Rimeannian manifold with $H_7(M, \ZZ_2) = 0$. Then for any $C^k$-generic ($k \geq 4$ or $k = \infty$) Riemannian metric $g$, there exists a closed embedded smooth minimal hypersurface $\Sigma$ in $(M,g)$.
    \end{Thm}

    \begin{Thm}
        Every $8$-dimensional closed Riemannian manifold $M$ with any $C^\infty$-generic metric admits a closed smooth minimal hypersurface $\Sigma$.
    \end{Thm}

    \begin{proof}
        If $H_7(M, \ZZ_2) \neq 0$, then we can take a nontrivial homology class $\alpha \in H_7(M, \ZZ_2)$. The standard Federer-Fleming minimizing process \cite{federerNormalIntegralCurrents1960} implies the existence of an area-minimizing multiplicity one minimal hypersurface $\Sigma \subset (M, g)$ with $[\Sigma] = \alpha$. In this case, by Lemma \ref{lem:surgery} below, one can show that there exists a metric $g'$ arbitrarily close to $g$ such that $(M, g')$ admits a smooth non-degenerate minimal hypersurface $\Sigma'$ close to $\Sigma$ and $[\Sigma'] = \alpha$ as well. The openness of the set of such $g'$'s follows immediately from the non-degenerateness and B. White's structure theorem \cite{whiteSpaceMinimalSubmanifolds1991,whiteBumpyMetricsTheorem2017}.

        The $H_7(M, \ZZ_2) = 0$ case is exactly the main theorem.
    \end{proof}

    Our study only deals with minimal hypersurfaces in a close Riemannian manifold, but in a general setting, one may also be interested in generic regularity of minimal submanifolds with codimension greater than $1$. In this direction, B. White \cite{whiteGenericRegularityUnoriented1985,whiteGenericTransversalityMinimal2019} has proved generic smoothness and embeddedness of minimizing integral $2$-cycles and J.D. Moore \cite{mooreBumpyMetricsClosed2006,mooreCorrectionBumpyMetrics2007} has shown generic nonexistence of branch points for parametrized minimal surfaces.

    Recently, N. Edelen \cite{edelenDegeneration7dimensionalMinimal2021} proved that in a closed Riemannian manifold $(M^8, g)$, the number of diffeomorphism classes of minimal hypersurfaces of bounded area and Morse index is finite. More precisely, he proved the existence of model minimal hypersurfaces such that others are bi-Lipschitz to these models. This suggests the possibility to describe the structure of singular minimal hypersurfaces as in B. White's structure theorem \cite{whiteSpaceMinimalSubmanifolds1991}, and inspires us to conjecture that the generic smoothness should hold for all minimal hypersurfaces with optimal regularity and finite Morse index. We expect to solve the $8$-dimensional case in the near future.

    \subsection{Some conventions}
        Throughout this paper, unless otherwise stated, a minimal hypersurface $\Sigma$ in an $(n+1)$-dimensional closed manifold $M$ is referred to \textit{a smooth, locally stable $n$-dimensional submanifold with locally bounded area and optimal regularity}, i.e. $\cH^n(\Sigma \cap K)<+\infty $ for every compact subset $K\subset M$, and $\cH^{n-2}(\overline{\Sigma}\setminus \Sigma) = 0$. These are exactly the minimal hypersurfaces generated from area minimizing arguments or Almgren-Pitts min-max theory (See \cite{pittsExistenceRegularityMinimal1981,schoenRegularityStableMinimal1981, liExistenceInfinitelyMany2019}). We always identify $\Sigma$ with the regular part of $\overline{\Sigma}$, and denote $\overline{\Sigma}\setminus \Sigma$ by $\Sing(\Sigma)$. 

        For a Caccioppoli set $A$, we shall identify $A$ with $\cH^{n+1}$-density $1$ part of $A$. We will use $\partial^t A$ to denote its topological boundary to avoid confusion with the boundary map defined in Subsection \ref{subsect:one-para}.

    \subsection{Sketch of Proof}
        We define 
        \begin{align*}
            \scG &:= \set{C^k \text{ Riemannian metrics on } M}\,;\\
            \scG_F &:= \set{g \in \scG| (M, g) \text{ has Frankel property}}\,;\\
            \scG_{NF} &:= \scG\backslash \scG_F\,;\\
            \scR &:= \set{g \in \scG| (M, g) \text{ admits a nondegenerate smooth minimal hypersurface}}\,.
        \end{align*}
        To prove our main theorem, it suffices to show that $\scR$ is open and dense in $\scG$. As one will see later, the most difficult part is to justify the denseness of $\scR$ in the interior of $\scG_F$, i.e., $\mathrm{int}(\scG_F)$.
        
        To prove this density, we introduce in general a geometric invariant, \textbf{singular capacity}, denoted by $\SCap$, on the space of $8$-dimensional Riemannian manifolds paired with a locally stable minimal hypersurface. Roughly speaking, this invariant counts how many singularities ``potentially'' a minimal hypersurface contains. The key of this invariant is its upper semi-continuity:
        \[
           \SCap(\Sigma; U, g) \geq \limsup_{j\to \infty} \SCap(\Sigma_j; U, g_j)   \]
        where $g_j \rightarrow g$ in $C^k_{\mathrm{loc}}$ and $\Sigma_j \rightarrow \Sigma$ in the varifold sense with some technical assumptions.

        With the notion of singular capacity, the proof can be decomposed into three steps.\\
        
        \paragraph{\textbf{Step 1}} 
            For every $g\in \mathrm{int}(\scG_F)$, let $V = \kappa|\Sigma| \in \cR$ be a $1$-width min-max integral varifold generated from an ONVP sweepout. One can observe the following dichotomy:
            \begin{itemize}
                \item either $\mathfrak{h}_{nm}(\Sigma) = \emptyset$;
                \item or $V = |\Sigma|$ and $\Sigma$ is connected, two-sided and separating with isolated singular points $\mathrm{Sing}(\Sigma)=\{p_0, p_1, ..., p_k\}$.
            \end{itemize}

            Let's consider the latter case. By constructing a metric perturbation by hand, we may further assume that $\Sigma$ is the unique and non-degenerate one realizing $1$-width provided that $\cH^0(\mathfrak{h}_{nm}(\Sigma)) = 1$. Let $\nu$ denote a unit normal on $\Sigma$.\\

        \paragraph{\textbf{Step 2}} 
            Let $f \in \scF \subset C^\infty_c(M - \mathrm{Sing}(\Sigma))$ be a generically chosen function with $\nu(f)|_{\Sigma} \not\equiv 0$. Let the metric perturbation be $g_t:= g(1+ tf)$ and $\Sigma_t$ be a $1$-width min-max minimal hypersurfaces in $(M, g_t)$. 

            It follows immediately from the uniqueness of $\Sigma$ that $\Sigma_t \rightarrow \Sigma$ in the varifold sense. By analyzing the associated Jacobi field $u\cdot\nu$ on $\Sigma$ generated by $\set{\Sigma_t}$ as in \cite[Section~4]{wangDeformationsSingularMinimal2020}, we can show that at some $p \in \mathrm{Sing}(\Sigma)$, the asymptotic rate of $u$ (See Subsection \ref{subsect:Jaccobi}) is either $\gamma^+_1(C_p)$ or $\gamma^-_1(C_p)$. This implies that $\Sigma_t$ is smooth in a small neighborhood $U_p$ of $p$ for $t$ close to $0$.
            
            A naive induction argument based on the $\cH^0(\Sing)$ would not work, because in a neighborhood $U_q$ of another singular point $q \in \overline \Sigma$, $\Sigma_t$ might have more than $1$ singular point inside. Fortunately, by upper semi-continuity of the singular capacity, we can conclude that 
            \[
                \SCap(\Sigma_{t_j}; M, g_{t_j}) \leq \SCap(\Sigma; M, g) - 1\,.   
            \]\\

        \paragraph{\textbf{Step 3}} 
            Iterate Step 1 and Step 2, and then a backward induction argument would lead to the following dichotomy. For some $\tilde g \in \mathrm{int}(\scG_F)$ arbitrarily close to $g$, we have a minimal hypersurface $\Sigma \subset (M, \tilde g)$ satisfying
            \begin{itemize}
              \item either $\mathfrak{h}_{nm}(\Sigma) = \emptyset$;
              \item or $\SCap(\Sigma) = 0$.
            \end{itemize}

            In the former case, Hardt-Simon type metric perturbation would lead to a smooth minimal hypersurface as in \cite{chodoshSingularBehaviorGeneric2020}; In the latter case, by definition, $\Sigma$ itself is smooth. \\

        In summary, we conclude the main theorem in $\mathrm{int}(\mathscr{G}_F)$.

\section*{Acknowledements}
We are grateful to our advisor Fernando Cod\'a Marques for his constant support. The second author would like to thank Xin Zhou for introducing this problem to him and his interest in this work. We would also thank Otis Chodosh, Yevgeny Liokumovich and Luca Spolaor for having a discussion and sharing their ideas on \cite{chodoshSingularBehaviorGeneric2020} with us.

\section{Preliminaries}
    We mainly focus on a closed Riemannian manifold $(M^{n+1}, g)$ with $H_n(M, \ZZ_2) = 0$. In this case, by Poincaré duality, its first cohomology group with $\ZZ_2$ coefficients $H^1(M, \ZZ_2) = 0$. Therefore, the Stiefel-Whitney class $\omega_1(M)$ vanishes and $M$ is orientable.

    Let's first list some notations in geometric measure theory and Almgren-Pitts' min-max theory. Interested readers may refer to L. Simon's lecture notes \cite{simonLecturesGeometricMeasure1984} and J. Pitts' monograph \cite{pittsExistenceRegularityMinimal1981}.
    
    \begin{itemize}
        \item $\cH^n$: $n$-dimensional Housdorff measure;
        \item $\mbfI_k(M;\ZZ_2)$: the space of $k$-dimensional mod $2$ flat chains in $M$;
        \item $\cZ_k(M;\ZZ_2)$: the space of $k$-dimensional mod $2$ flat cycles in $M$;
        \item $\mbfM_U$: the mass norm in an open subset $U$ on the flat chain space ($\mbfM_M$ abbreviated to $\mbfM$);
        \item $\cF$: the flat metric on the flat chain space defined by 
        \[\cF(S,T) = \inf\set{\mbfM(R) + \mbfM(P): S - T = \partial R + P}\,,\]
        where $S, T, P \in \mbfI_k(M; \ZZ_2)$ and $R \in \mbfI_{k+1}(M; \ZZ_2)$. 
        We always assume that $\mbfI_n(M; \ZZ_2)$ is endowed with $\cF$ metric, and so is $\cZ_n(M; \ZZ_2)$;
        \item $\cC(M)$: the space of Caccioppoli sets in $M$, i.e., subsets of $M$ of finite perimeter. The metric on $\cC(M)$ is the metric induced by the symmetric difference, i.e., for a pair of sets $X, Y \in \cC(M)$,
            \begin{equation*}
                d_{\cC(M)}(X, Y) = \cH^{n+1}(X\Delta Y)\,.
            \end{equation*}
        \item $\cV_n(M)$: the space of $n$-dimensional varifolds in $M$;
        \item $\IV_n(M)$: the space of $n$-dimensional integral varifolds in $M$;
        \item $\cR(M)$: the space of $n$-dimensional integral varifold in $M$ whose support is regular away from a closed singular set of Hausdorff dimension $\leq n - 7$;
        \item $|T|$: the associated integral varifold for $T\in \mbfI_n(M; \ZZ_2)$;
        \item $\|T\|$: the associated Radon measure for $T \in \mbfI_n(M; \ZZ_2)$;
        \item $\mbfF$: the metric on $\cV_n(M)$ defined by 
            \[
                \mbfF(V, W) = \sup\set{\|V\|(f) - \|W\|(f): f \in C^1(M), |f|\leq 1, |Df|\leq 1}\,.
            \]
            Moreover, $\mbfF_U(V, W) := \mbfF(V_1\llcorner U, V_2\llcorner U)$ for any open subset $U$;
        \item $\|V\|$: the associated Radon measure for $V \in \IV_n(M)$.
    \end{itemize}

    \subsection{Caccioppoli Sets}

        In this subsection, we shall briefly revisit Caccioppoli sets, i.e., sets of finite perimeter in the Riemannian manifold setting. More details in the Euclidean setting can be found in L. Simon's lecture notes \cite{simonLecturesGeometricMeasure1984} and multiple monographs written by H. Federer \cite[Chap. 4]{federerGeometricMeasureTheory1969}, F. Lin and X. Yang \cite[Chap. 5]{linGeometricMeasureTheory2002}, F. Maggi \cite{maggiSetsFinitePerimeter2012a}, L.C. Evans and R.F. Gariepy \cite{evansMeasureTheoryFine2015}, E. Giusti \cite{giustiMinimalSurfacesFunctions1984}, etc., whose statements can be easily extended to our setting.

        \begin{Def}
            A Lebsgue measurable set $E \subset (M^{n+1}, g)$ is \textbf{Caccioppoli} (have finite perimeter) if its charateristic function $\chi_E \in \BV(M)$. In other words, $\chi_E \in L^1(M)$ has bounded variation, i.e., 
            \begin{equation}
                \|D_g\chi_E\|(M) := \sup\set{\int_E \operatorname{div}_g X \mathrm{d}x: X \in \Gamma^1(TM), |X|_g \leq 1} < \infty\,.
            \end{equation}
            Note here $D_g\chi_E: \Gamma^1(TM) \rightarrow \mathbb{R}$ is the associated functional with $E$. Two sets will be identified if they differ by a measure zero set and the collection of all the Caccioppoli sets are denoted by $\mathcal{C}(M)$.
        \end{Def}
        \begin{Rem}\label{rem:equiv}
            It is easy to verify that $\cC(M)$ is independent of the choice of $C^\infty$ metrics on $M$. 
        \end{Rem}

        Since $M$ is assumed to be orientable, we can take a unique $(n+1)$-vector $\xi$ in $\Lambda^{n+1} TM$ so that at each point $p \in M$, $\xi_p$ is the wedge product of an orthonormal basis of $T_pM$. We will use $\cL^{n+1}$ to denote the Lebesgue measure on $M$ and $\mathbf{E}^{n+1}$ to denote the current $\cL^{n+1}\wedge\xi$. By definition, a set $E \in \cC(M)$ is equivalent to $\mathbf{E}^{n+1}\llcorner E \in \mbfI_{n+1}(M;\ZZ_2)$.

        \begin{Def}
            Given $E \in \cC(M)$, by Riesz's representation theorem, there is a unique $TM$-valued Radon measure $\mu_E$ so that
            \begin{equation}
                \int_E \operatorname{div}_g X = \int_M X \cdot_g \mathrm{d}\mu_E, \quad \forall X \in \Gamma^1(TM)\,.
            \end{equation}
            Moreover, $|\mu_E|(M)$ is finite.
            $\mu_E$ is called the \textbf{Gauss-Green measure} of $E$. And we define the \textbf{relative perimeter} of $E$ in $F \subset M$, and the \textbf{perimeter} of $E$ to be
            \begin{equation}
                P(E; F) = |\mu_E|(F)\,, \quad P(E) = |\mu_E|(M)\,.
            \end{equation}
        \end{Def}
        \begin{Rem}
            If $E \subset (M^{n+1}, g)$ is Caccioppoli, then so is $M \setminus E$. Furthermore, 
            \begin{equation}
                \mu_{M \setminus E} = -\mu_E\,, \quad P(E) = P(M \setminus E)\,.
            \end{equation}
        \end{Rem}

        \begin{Def}
            Given $E \subset (M^{n+1}, g)$ and $x \in M$, if the limit
            \begin{equation}
                \theta_{n+1}(E)(x) = \lim_{r \rightarrow 0^+} \frac{|E \cap B_g(x, r)|}{|B_g(x,r)|}
            \end{equation}
            exists, then $\theta_{n+1}(E)(x)$ is called the \textbf{$(n+1)$-dimensional density of $E$ at $x$}. Given $t \in [0, 1]$, the \textbf{set of points of density $t$ of $E$} is defined to be
            \begin{equation}
                E^{(t)} := \set{x \in M: \theta_{n+1}(E)(x) = t}\,.
            \end{equation}
        \end{Def}
        
        \begin{Def}
            Given $E \in \cC(M)$, the \textbf{reduced boundary} $\partial^* E$ is the set of those $x \in \supp \mu_E$ such that
            \begin{equation}
                \nu_E = \lim_{r \rightarrow 0^+} \frac{\mu_E(B(x, r))}{|\mu_E|(B(x, r))} \quad \text{exists and belongs to } S^{n} \,.
            \end{equation}
            Here $\nu_E$ is called the \textbf{measure-theoretic outer unit normal} to $E$.
        \end{Def}
        \begin{Rem}
            By Lebesgue-Besicovitch differentiation theorem, $\mu_E = \nu_E |\mu_E|\llcorner \partial^* E$, where $\nu_E$ is a $|\mu_E|$ measurable vector field.
        \end{Rem}
        \begin{Rem}\label{rem:red_topo_bdry}
            One can check that $\partial^* E \subset \overline{\partial^* E} = \supp \mu_E \subset \partial^t E$, where the last one is the topological boundary of $E$. Up to modification on sets of measure zero, we also have
            \begin{equation}
                \overline{\partial^* E} = \partial^t E\,.
            \end{equation}

            A Caccioppoli set is in fact an equivalence class of sets differed by a measure zero set, so for convenience, we will only consider all the points of density $1$ which satisfies the inequality above for each Caccioppoli set.
        \end{Rem}

        \begin{Def}
            For a set $E \subset (M^{n+1}, g)$, we define its \textbf{essential boundary} $\partial^e E$ to be
            \begin{equation}
                \partial^e E = M \setminus (E^{(0)} \cup E^{(1)})\,.
            \end{equation}
        \end{Def}

        In the monograph, H. Federer gave a criterion for Caccioppoli sets {\cite[Theorem~4.5.11]{federerGeometricMeasureTheory1969}}, and here we adapt it to the Riemannian manifold setting.

        \begin{Thm}
            For $E \subset (M^{n+1}, g)$, the following two conditions are equivalent:
            \begin{enumerate}
                \item $E$ is Lebesgue measurable and $\partial(\mathbf{E}^{n+1} \llcorner E)$ is representable by integration, i.e., $E$ is a Caccioppoli set.
                \item $\mathscr{I}^n\left(\partial^e E\right) < \infty$.
            \end{enumerate}
        \end{Thm}
        \begin{Rem}
            By Nash embedding theorem, we can take a large integer $N$ such that $(M, g)$ can be isometrically embedded in $\mathbb{R}^N$. Then $\mathscr{I}^n$ denotes the \textbf{$n$ dimensional integral geometric measure} over $\mathbb{R}^N$ (Interested readers may refer to \cite[2.10.5, 2.10.15, 2.10.16]{federerGeometricMeasureTheory1969} or \cite[2.4]{morganGeometricMeasureTheory2016} for more details).
        \end{Rem}

        \begin{Cor}\label{cor:sep_cacc}
            Suppose that a minimal hypersurface $\Sigma \subset (M, g)$ separates $M$ into two open connected components $M_+$ and $M_-$, and then both $M_+$ and $M_-$ are Caccioppoli sets.
        \end{Cor}
        \begin{proof}
            It suffices to check that $\mathcal{I}^n\left(\partial^e M_\pm\right) < \infty$ and we only focus on $M_+$. By definition, we have $M_+ \subset M^{(1)}_+$ and $M_- \subset M^{(0)}_+$. Since $\Sigma$ is smooth almost everywhere, we can conlude that $\partial^e M_+ \approx \Sigma$, and $\mathscr{I}^n\left(\partial^e E\right) = \mathcal{H}^n(\Sigma) < \infty$. Hence, $M_+$ is a Caccioppoli set.
        \end{proof}

        Before concluding this subsection, we state an important theorem showing that $\cC(M)$ is close under set operations and characterizing the Gauss-Green measure of set operations.

        \begin{Thm}[{\cite[Theorem~16.3]{maggiSetsFinitePerimeter2012a}}]\label{thm:char}
            Given $E, F \in \cC(M)$, let
            \begin{align*}
                \set{\nu_E = \nu_F} &= \set{x \in \partial^*E \cap \partial^*F: \nu_E(x) = \nu_F(x)}\,,\\
                \set{\nu_E = -\nu_F} &= \set{x \in \partial^*E \cap \partial^*F: \nu_E(x) = -\nu_F(x)}\,.
            \end{align*}

            Then $E \cap F$, $E \setminus F$ and $E \cup F$ are all Caccioppoli, with
            \begin{align*}
                \mu_{E \cap F} &= \mu_E\llcorner F^{(1)} + \mu_F \llcorner E^{(1)} + \nu_E \mathcal{H}^n \llcorner\set{\nu_E = \nu_F}\,,\\
                \mu_{E \setminus F} &= \mu_E\llcorner F^{(0)} - \mu_F \llcorner E^{(1)} + \nu_E \mathcal{H}^n \llcorner\set{\nu_E = -\nu_F}\,,\\
                \mu_{E \cup F} &= \mu_E\llcorner F^{(0)} + \mu_F \llcorner E^{(0)} + \nu_E \mathcal{H}^n \llcorner\set{\nu_E = \nu_F}\,,
            \end{align*}
            and in the measure-theoretic sense,
            \begin{align*}
                \partial^*(E \cap F) &\approx (F^{(1)}\cap \partial^*E) \cup (E^{(1)}\cap \partial^*F) \cup \set{\nu_E = \nu_F}\,,\\
                \partial^*(E \setminus F) &\approx (F^{(0)}\cap \partial^*E) \cup (E^{(1)}\cap \partial^*F) \cup \set{\nu_E = -\nu_F}\,,\\
                \partial^*(E \cup F) &\approx (F^{(0)}\cap \partial^*E) \cup (E^{(0)}\cap \partial^*F) \cup \set{\nu_E = \nu_F}\,.
            \end{align*}

            Moreover, for every borel set $G \subset M$,
            \begin{align*}
                P(E \cap F; G) &= P(E; F^{(1)}\cap G) + P(F; E^{(1)} \cap G) + \cH^{n}(\set{\nu_E = \nu_F}\cap G)\,,\\
                P(E \setminus F; G) &= P(E; F^{(0)}\cap G) + P(F; E^{(1)} \cap G) + \cH^{n}(\set{\nu_E = -\nu_F}\cap G)\,,\\
                P(E \cup F; G) &= P(E; F^{(0)}\cap G) + P(F; E^{(0)} \cap G) + \cH^{n}(\set{\nu_E = \nu_F}\cap G)\,.
            \end{align*}
        \end{Thm}

    \subsection{1-parameter Min-max Theory}\label{subsect:one-para}

        Let's consider the boundary map $\partial: \cC(M) \rightarrow \cZ_n(M; \ZZ_2)$: More precisely, for any $A \in \cC(M)$, we use $\partial A$ to denote the mod $2$ flat cycle $\partial(\mathbf{E}^{n+1}\llcorner A)$, where $\mathbf{E}^{n+1}$ is defined in the previous subsection. It's easy to see that
        \begin{equation}
            P(A) = \mbfM(\partial A)\,.
        \end{equation}

        Since we assume that $H_n(M, \ZZ_2) = 0$, the boundary map is surjective. By Constancy Theorem, this map is in fact a double cover. With this observation, we can utilise Caccioppoli sets instead of mod $2$ flat cycles to define sweepouts on $M$. This has been used in \cite{zhouMinMaxHypersurface2017,zhouMinmaxTheoryConstant2017,zhouExistenceHypersurfacesPrescribed2018,zhouMultiplicityOneConjecture2019,chodoshSingularBehaviorGeneric2020}.
        
        \begin{Def}
            A (1-parameter) \textbf{sweepout} on $M$ is a continuous map $\Phi: [0,1] \rightarrow \cC(M)$ with $\Phi(0) = 0$ and $\Phi(1) = M$. The collection of all sweepouts is denoted by $\cP$.
        \end{Def}
        
        \begin{Def}
            The \textbf{min-max width} on $M$ is defined to be
            \begin{equation*}
                \cW = \inf_{\Phi \in \cP} \sup_{t \in [0,1]} \mbfM(\partial\Phi(t))>0\,.
            \end{equation*}
        \end{Def}
        
        The Almgren-Pitts theory \cite{almgrenTheoryVarifoldsVariational1965,pittsExistenceRegularityMinimal1981,schoenRegularityStableMinimal1981} has shown that the width $\cW$ could always be realized by a minimal hypersurface (possibly with multiplicities). Recently, O. Chodosh, Y. Liokumovich and L. Spolaor \cite{chodoshSingularBehaviorGeneric2020} gave a refined description of such a minimal hypersurface in terms of its singularities and Morse index. Here, we adapt their results to our setting.

        \begin{Def}[ONVP sweepouts]
            A sweepout $\Phi$ is called an \textbf{optimal nested volume parametrized (ONVP) sweepout} if it satisfies the following conditions.
            \begin{description}
                \item[1. optimal] $\sup_{x\in[0,1]} \mbfM(\partial\Phi(x)) = \cW$;
                \item[2. nested] $\Phi(x_1) \subset \Phi(x_2)$, for all $0 \leq x_1 \leq x_2 \leq 1$;
                \item[3. volume parametrized] $\mathrm{Vol}(\Phi(x)) = x \cdot \mathrm{Vol(M, g)}$, for every $x \in [0, 1]$. 
            \end{description}

            The \textbf{critical domain} of $\Phi$ is the set
            \begin{equation}
                \mathbf{m}(\Phi) = \set{x \in [0, 1]: \limsup_{y \rightarrow x}\mathbf{M}(\partial\Phi(y)) = \cW}\,.
            \end{equation}
            Similarly, the \textbf{left} (resp. \textbf{right}) \textbf{critical domain} of $\Phi$ can be defined as the set $\mathbf{m}_L(\Phi)$ (resp. $\textbf{m}_R(\Phi)$) only involving $y \nearrow x$ (resp. $y \searrow x$). Apparently, $\mathbf{m}(\Phi) = \mathbf{m}_L(\Phi) \bigcup \mathbf{m}_R(\Phi)$.
            
            The \textbf{critical set} of $\Phi$ is 
            \begin{align}
             \mbfC(\Phi):=  \{V\in \cV_n(M): V = \lim_{j\to \infty} |\partial \Phi(x_j)|,\ x_j\in [0,1];\ \|V\|(M) = \cW\}.
            \end{align}
        \end{Def}

        \begin{Def}[Excessive points]
            A point $x_0 \in [0,1]$ is called \textbf{left} (resp. \textbf{right}) \textbf{excessive} for a sweepout $\Phi$, if there exists a constant $\varepsilon > 0$ and an interval $I = [a, b]$, $[a, b)$, $(a, b]$ or $(a, b)$ with $(x_0 - \varepsilon, x_0] \subset I$ (resp. $[x_0, x_0 + \varepsilon) \subset I$), satisfying the following replacement condition.

            We can find a continuous map on $I$, $\set{\Phi^I(x)}_{x\in I}$, such that $\Phi^I(a) = \Phi(a)$ and $\Phi^I(b) = \Phi(b)$ but for all $x \in I$,
            \begin{equation}
                \limsup_{I \ni y \rightarrow x} \mathbf{M}(\partial\Phi^I(y)) < W\,.
            \end{equation}
        \end{Def}

        \begin{Thm}[{\cite[Theorem~19]{chodoshSingularBehaviorGeneric2020}}]\label{thm:ONVP}
            For any closed Riemannian manifold $(M, g)$, there exists an (ONVP) sweepout $\Psi$ such that every $x \in \mathbf{m}_L(\Psi)$ is not left-excessive and every $x \in \mathbf{m}_R(\Psi)$ is not right excessive. The Almgren-Pitts min-max theory implies that there exists a stationary integral varifold $V \in \mathbf{C}(\Psi)$ whose support is the closure of a minimal hypersurface $\Sigma$, and $\|V\|(M) = \cW$.
        \end{Thm}

        \begin{Def}[One-sided homotopy area-minimizing]
            Given a minimal hypersurface $\Sigma \subset (M, g)$, $p \in \overline \Sigma$ and $r > 0$ small enough such that $\overline \Sigma \cap B_r(p)$ separates the open ball $B_r(p)$ into two open connected components $E_+$ and $E_-$. $\Sigma$ is said to be \textbf{one-sided homotopy area-minimizing (OSHAM)} in $B_r(p)$, if there does not exist a deformation $\set{\Omega(t) \subset \cC(M)}_{t \in [0, 1]}$  satisfying the following conditions.
            \begin{enumerate}
                \item $\Omega(0)\cap B_r(p) = E_\pm$ and $\Omega(t) \subset \Omega(s)$ for any $t \geq s$;
                \item $\partial^* \Omega(t) \Delta \partial^* E_\pm \subset B_r(p)$;
                \item $\mathbf{M}_{B_r(p)}(\partial \Omega(t)) \leq \mathbf{M}_{B_r(p)}(\partial E_\pm)$ and $\mathbf{M}_{B_r(p)}(\partial \Omega(1)) < \mathbf{M}_{B_r(p)}(\partial E_\pm)$.
            \end{enumerate}

            We define
            \begin{equation}
                \mathfrak{h}_{nm}(\Sigma) := \set{p \in \overline \Sigma: \forall r > 0 \text{ small, } \Sigma\cap B_r(p) \text{ is not (OSHAM) in }B_r(p)}\,.
            \end{equation}
        \end{Def}

        \begin{Thm}[{\cite[Theorem~4, Proposition~29, Lemma~30]{chodoshSingularBehaviorGeneric2020}}]\label{thm:sing_index}
            Given a closed Riemannian manifold $(M^8, g)$ and a sweepout $\Psi$ in Theorem \ref{thm:ONVP}, by possibly replacing $\Phi(x)$ by $M \setminus \Phi(1-x)$, there exists a sequence $x_i \nearrow x_0 \in \mathbf{m}_L(\Phi)$ such that 
            \[
                \lim_i |\partial \Phi(i)| \rightarrow V = \sum_i \kappa_i|\Sigma_i| \in \mathbf{C}(\Phi)\,,
            \]
            where $\Sigma_i$'s are pairwise disjoint minimal hypersurfaces with $\kappa_i \in \set{1, 2}$ and
            \begin{equation}
                \sum_i\cH^0(\mathfrak{h}_{nm}(\Sigma_i)) + \sum_i\mathrm{Index}(\Sigma_i) \leq 1\,.
            \end{equation}
            
            Moreover, if there exists a $\kappa_i = 2$, then 
            \begin{equation}
                \mathfrak{h}_{nm}(\Sigma) = \emptyset\,.    
            \end{equation}
            Otherwise, $V$ is of multiplicity one and $V = |\partial \Phi(x_0)|$.
        \end{Thm}

        One essential ingredient of the proof is the following interpolation lemma which follows from a result of K.J. Falconer \cite{falconerContinuityPropertiesKplane1980} (See also \cite[Appendix~6]{guthVolumesBallsLarge2011}, \cite[Lemma~5.3]{chambersExistenceMinimalHypersurfaces2020}).

        \begin{Lem}[{\cite[Lemma~16]{chodoshSingularBehaviorGeneric2020}}]\label{lem:interpolation}
            On $(M^{n+1}, g)$, for every $L, \varepsilon > 0$, there exists a $\delta > 0$ satisfying the following property.

            For any $\Omega_0, \Omega_1 \in \cC(M)$ with $\Omega_0 \subset \Omega_1$, $P(\Omega_i) \leq L$ and $\mathrm{Vol}(\Omega_1 \setminus \Omega_0) \leq \delta$, there exists a nested continuous family $\set{\Omega_t}$ with
            \begin{equation}
                P(\Omega_t) \leq \max\set{P(\Omega_1), P(\Omega_0)} + \varepsilon\,,
            \end{equation}
            for all $t \in [0, 1]$.
        \end{Lem}

    \subsection{Surgery Procedure \`a la Chodosh-Liokumovich-Spolaor}

        In this subsection, we will recall the surgery procedure in \cite{chodoshSingularBehaviorGeneric2020} to perturb away singularities locally when $\mathfrak{h}_{nm}(\Sigma) = \emptyset$ for a minimal hypersurface $\Sigma^7 \subset (M^8, g)$. Since the surgery was done locally, from now on, we focus on a singular point $p \in \mathrm{Sing}(\Sigma)$. WLOG, let's assume that for $\varepsilon_0 > 0$, $\Sigma$ has only one singular point and is OSHAM in $B_{2\varepsilon_0}(p)$.

        \begin{Lem}[Surgery Procedure {\cite[Proposition~31]{chodoshSingularBehaviorGeneric2020}}]\label{lem:surgery}
            Let $M, g, \Sigma, p$ and $\varepsilon_0$ be as above. For every $k\geq 4$, $\delta > 0$, there exists a Riemmanian metric $g'$ and a minimal hypersurface $\Sigma' \subset (M, g')$ satisfying the following conditions.
            \begin{enumerate}
                \item $\|g - g'\|_{C^k} < \delta$;
                \item $g = g'$ and $\Sigma = \Sigma'$ outside $B_{2\varepsilon_0}(p)$;
                \item $\mathrm{Sing}(\Sigma') \cap B_{2\varepsilon_0} = \emptyset$.
            \end{enumerate}
        \end{Lem} 

    \subsection{Associated Jacobi Fields}\label{subsect:Jaccobi}

        We shall recall some notions and some results from \cite{wangDeformationsSingularMinimal2020}, which will be utilised later.

        Let $\Sigma \subset (M^8, g)$ be a two-sided minimal hypersurface with a unit normal field $\nu$. It follows from \cite[Section~5.4]{federerGeometricMeasureTheory1969} and \cite{schoenRegularityStableMinimal1981,simonAsymptoticsClassNonLinear1983} that the singular set $\Sing(\Sigma):=\overline{\Sigma}\setminus \Sigma$ of $\Sigma$ consists of isolated points, at each of which $\Sigma$ has a unique and regular tangent cone.

        On $\Sigma$, the space of functions that we are going to work on will be denoted by $\scB(\Sigma)$ and defined as follows. By \cite[Lemma~3.1]{wangDeformationsSingularMinimal2020}, one can observe that there exists $C_\Sigma >0$ such that \[
            \|\phi\|^2_{\scB(\Sigma)} := Q_\Sigma(\phi, \phi ) + C_\Sigma\|\phi\|^2_{L^2(\Sigma)} \geq \|\phi\|_{L^2(\Sigma)}^2\quad \forall \phi\in C_c^1(\Sigma)\,,\]
        where $Q_\Sigma(\phi, \phi):= \int_\Sigma |\nabla \phi|^2 - (|A_\Sigma|^2 + Ric_M(\nu, \nu))\phi^2$ be the quadratic form associated to the second variation of area functional at $\Sigma$. Hence, \[\scB(\Sigma):= \overline{C_c^\infty(\Sigma)}^{\|\cdot\|_{\scB}}\,,\] is a well defined Hilbert space and is naturally embedded in $L^2(\Sigma)$. Moreover, by \cite[Lemma~3.2]{wangDeformationsSingularMinimal2020}, every $\phi\in \scB(\Sigma)$ is locally $W^{1,2}$ on $\Sigma$, and by \cite[Proposition 3.5 \& Lemma~3.9]{wangDeformationsSingularMinimal2020}, $\scB(\Sigma)\hookrightarrow L^2(\Sigma)$ is a compact embedding.

        With $\scB(\Sigma)$, we can define the Morse index via the Jacobi operator $L_\Sigma:= \Delta_\Sigma + |A_\Sigma|^2+Ric_M(\nu, \nu)$ associated to $Q_\Sigma$. Thanks to the compact embedding, we can define the $L^2$-eigenvalues and eigenfunctions for $-L_\Sigma$ and derive the spectral decomposition of $L^2(\Sigma)$ as well as $\scB(\Sigma)$. Recall that the \textbf{Morse index} of $\Sigma$ has been defined in \cite{marquesEquidistributionMinimalHypersurfaces2019,deyCompactnessCertainClass2019} as the maximal dimension of the linear subspace of smooth ambient vector fields decreasing its area functional at second order. With $\scB(\Sigma)$, an equivalent definition \cite[Corollary~3.7]{wangDeformationsSingularMinimal2020} could be \[
            \mathrm{Index}(\Sigma) = \sum_{\lambda_j < 0} \dim E_j\,,\]
        where $E_j\subset \scB(\Sigma)$ is the $j$-th eigenspace of $L_\Sigma$. 

        $\Sigma$ is called \textbf{non-degenerate} if $0$ is not an eigenvalue of $-L_\Sigma$. When $\Sigma$ is non-degenerate, by \cite[Proposition~3.5]{wangDeformationsSingularMinimal2020} for every $f\in L^2(\Sigma)$, the equation $L_\Sigma u = f$ has a unique solution $u\in \scB(\Sigma)$, denoted by $L^{-1}_\Sigma(f)$.

        Now, for each singular point $p \in \mathrm{Sing}(\Sigma)$ with $B_{r_p}(p)$ small enough  such that $-L_\Sigma$ is strictly positive on $\scB_0(B_{r_p}(p)): = \overline{C_c^\infty(B_{r_p}(p)\cap \Sigma)}^{\|\cdot\|_{\scB}}$, according to \cite[Subsection~3.3]{wangDeformationsSingularMinimal2020}, we can define a unique (up to a normalization) \textbf{Green's function} $G_p \in C^\infty_{\mathrm{loc}}(\overline B_{r_p}(p)\cap \Sigma)$ of $L_\Sigma $ which vanishes on $\partial B_{r_p}(p)\cap \Sigma$. We extend $G_p$ to $\Sigma$ by setting $G_p = 0$ outside $B_{r_p(p)}$.

        \begin{Lem}[{\cite[Theorem 4.2]{wangDeformationsSingularMinimal2020}}] \label{lemma:exist_Jac}
            Suppose that an $8$-dimensional closed Riemannian manifold  $(M, g)$ admits a minimal hypersurface $\Sigma$ with a unit normal $\nu$. Let $f$ be a smooth function defined on $M$ such that $\nu(f)\vert_\Sigma \not\equiv 0$, $\set{c_j}$ be a sequence of positive real numbers with $c_j \rightarrow 0$ and $\set{f_j}$ be a sequence of smooth functions with $f_j \rightarrow f$ in $C^4$. Let's further assume that for each metric $g_j := (1 + c_j f_j) g$, there exists a minimal surface $\Sigma_j \subset (M, g_j)$ with $\mathrm{Index}(\Sigma_j) = \mathrm{Index}(\Sigma)$ and $\Sigma_j \rightarrow \Sigma$ in the varifold sense with multiplicity $1$.

            Then after passing to a subsequence, there exists a generalized Jacobi field $0\neq u\in C^2(\Sigma)$ associated to the subsequence which will still be denoted by $\{\Sigma_j\}_{j\geq 1}$. More precisely, there exist functions $u_j\in C^2(\Sigma)$ and positive real numbers $t_j\to 0_+$ such that
            \begin{itemize}
                \item for every open subset $W\subset \subset M\setminus \Sing(\Sigma)$ and $j$ sufficiently large, 
                    \[
                        \mathrm{graph}_\Sigma (u_j) \cap W = \Sigma_j \cap W\,;
                    \]
                \item $u_j/t_j \to u$ in $C^2_{loc}(\Sigma)$; 
                \item $L_\Sigma u = c\nu(f)$ for some real number $c\geq 0$;
                \item \[u \in \scB(\Sigma)\oplus \RR_{L_\Sigma}\langle Sing(\Sigma) \rangle := \scB(\Sigma)\oplus \bigoplus_{p\in \mathrm{Sing}(\Sigma)}\RR  G_p \,.\]
            \end{itemize}
        \end{Lem}

        The Jacobi field generated above could help us understand the behavior of $\Sigma_j$ near $\Sigma$. In particular, it depicts a picture where generically, one of the singular points of $\Sigma$ can be perturbed away as $\Sigma_j$.

        \begin{Lem}  \label{Lem_Pre_Perturb to find MH reg}
            In Lemma \ref{lemma:exist_Jac}, if we further assume that $\Sigma$ is nondegenerate, then there exists an open dense subset $\scF\subset C_c^\infty(M\setminus \Sing(\Sigma))$ depending only on $M, g$ and $\Sigma$ with the following property. 

            For every $f\in \scF$ and every sequence $c_j \rightarrow 0_+$, if $(M, g_j)$ admits $\Sigma_j$ as described above, then there exists a small neighborhood $U_p\subset M$ of some $p\in \Sing(\Sigma)$ such that $\Sing(\Sigma_j) \cap U_p = \emptyset$ for infinitely many $j$.
        \end{Lem}
        \begin{proof}
            Let $u$ be an associated generalized Jacobi fields generated in the previous lemma. It follows from \cite[Corollary 3.15 \& 3.17]{wangDeformationsSingularMinimal2020} that the asymptotic rate of $u$ at $p\in \mathrm{Sing}(\Sigma)$ satisfies \[
             \cA\cR_p(u):= \sup\set{\sigma: \lim_{t \rightarrow 0_+} \int_{A_{t, 2t}(p) \cap \Sigma} u^2(x)\mathrm{dist}(x, p)^{-n-2 \sigma} = 0} \geq \gamma_1^-(C_p)\,,\]
            where $C_p$ is the tangent cone of $\Sigma$ at $p$, $\gamma_1^-(C_p)$ is a growth rate spectrum for Jacobi field on $C_p$ (See also \cite{simonIsolatedSingularitiesMinimal1982,caffarelliHardtSimon1984,hardtAreaMinimizingHypersurfaces1985}). 

            If $\cA\cR_p(u) > \gamma_1^+(C_p)$ for every $p\in \Sing(\Sigma)$, so then again by \cite[Corollary 3.15 \& 3.17]{wangDeformationsSingularMinimal2020}, we have $u\in \scB(\Sigma)$. Since $L_\Sigma$ is non-degenerate, then $u$ is the unique solution of $L_\Sigma u = c\nu(f)$ in $\scB(\Sigma)$ with $c\neq 0$. However, by \cite[Lemma~3.21]{wangDeformationsSingularMinimal2020}, the set $\scE$ of $f\in C_c^\infty(M\setminus \Sing(\Sigma))$ such that $\cA\cR_p(L_\Sigma^{-1}(\nu(f))) >\gamma_1^+(C_p)$ for some $p\in \Sing(\Sigma)$ is nowhere dense in $C_c^\infty(M\setminus \Sing(\Sigma))$. Hence, as long as we choose $\scF:= C_c^\infty(M\setminus \Sing(\Sigma)) \setminus \overline\scE$, this case could not happen.

            Therefore, for $f \in \scF$, there exists a singular point $p$ at which $\cA\cR_p(u)\leq \gamma_1^+(C_p)$, i.e., either $\cA\cR_p(u) = \gamma_1^-(C_p)$ or $\cA\cR_p(u) = \gamma_1^+(C_p)$ (\cite[Lemma~3.14]{wangDeformationsSingularMinimal2020}). Then it follows from \cite[Corollary~4.12]{wangDeformationsSingularMinimal2020} that for some neighborhood $U_p\supset p$, $\mathrm{Sing}(\Sigma_j)\cap U_p = \emptyset$ for infinitely many $j$.
        \end{proof}


\section{Generation of Candidate Minimal Hypersurfaces}\label{sec:gen}
    
    In this section, we shall discuss how to generate a candidate minimal hypersurface in a given Riemannian manifold $(M^{n+1}, g)$ with $H_n(M, \ZZ_2) = 0$. For simplicity, whenever it is clear, we shall abuse the use of set relations $=$ and $\subset$ for Caccioppoli sets in the measure-theoretic sense, i.e., up to a measure zero set.

    \subsection{Manifolds with Frankel Property}

        \begin{Def}
            A closed Riemannian manifold $(M^{n+1}, g)$ is said to have \textbf{Frankel property}, if any pair of minimal hypersurfaces has nonempty intersections.
        \end{Def}

        Theorem \ref{thm:ONVP} and Theorem \ref{thm:sing_index} together imply the existence of an (ONVP) sweepout $\Phi$ and a sequence $x_i \nearrow x_0 \in \mathbf{m}_L(\Phi)$ such that $|\partial \Phi(x_i)| \rightarrow V \in \cR$, where $V = \sum_i \kappa_i |\Sigma_i|$ for some pairwise disjoint minimal hypersurfaces $\Sigma_i$, where $\kappa_i \in \set{1, 2}$.

        If the ambient manifold $(M^8, g)$ has Frankel property, then $V = \kappa |\Sigma|$ for $\kappa \in \set{1, 2}$. Moreover, one of the following conditions holds:
        \begin{enumerate}
            \item either $\mathfrak{h}_{nm}(\Sigma) = \emptyset$;
            \item or $\mathfrak{h}_{nm}(\Sigma) \neq \emptyset$, $\kappa = 1$, and $\Sigma = \partial^*\Phi(x_0)$ is stable.
        \end{enumerate}

        In the second case, we will modify the sweepout such that near $\Sigma$ the mass of each slice is stricly smaller than $\operatorname{Area}(\Sigma)$, which provides a room for us to perturb the ambient metric without breaking the optimality of the sweepout.

        \begin{Lem}\label{lem:tech}
            For a closed ambient manifold $(M^8, g)$ with Frankel property, let $\Sigma$ be a minimal hypersurface generated from an (ONVP) sweepout $\Phi$ via $x_i \nearrow x_0 \in \mathbf{m}_L(\Phi)$. If $\mathfrak{h}_{nm}(\Sigma) = \set{p}$, we can construct a new (ONVP) sweepout $\Psi$ satisfying the following property.

            For any $r > 0$ small enough, there exists $\varepsilon_0 > 0$, an open set $U \supset \overline{\Sigma}$, and a compact set $K \subset U\cap B_r(p)$ containing $p$ such that for any $x \in [0, 1]$, we have
            \begin{equation}
                \mathbf{M}(\partial \Psi(x)) < \mathrm{Area}(\Sigma) - \varepsilon_0\,,
            \end{equation}
            provided that $\partial^* \Psi(x) \cap U \setminus (K \cup \overline{\Sigma}) \neq \emptyset$.
        \end{Lem}
        \begin{proof}
            By the definition of $\mathfrak{h}_{nm}(\Sigma)$, we can take a geodesic ball $B_r(p)$ with $r$ small enough such that $\partial B_r(p)$ is strictly convex, $\partial B_r(p) \cap \Sigma$ is a smooth codimension $2$ submanifold, and $\Sigma$ is not OSHAM on either side in $B_r(p)$.

            Since $H_7(M, \mathbb{Z}_2) = 0$, $\Sigma$ is two-sided and separates $M$ into two connected open component $M_+$ and $M_-$. By Corollary \ref{cor:sep_cacc}, they are both Caccioppoli sets. Because $\partial^*\Phi(x_0) = \Sigma$, by constancy theorem, w.l.o.g., we may assume that $\Phi(x_0) = M_+$ and thus, $\Phi(x) \subset M_+$ for $x \in [0, x_0]$. We will only focus on $M_+$, since the same process can be performed on $M_-$ as well.

            By the existence of homotopic minimizers \cite[Lemma~13]{chodoshSingularBehaviorGeneric2020} and the mean convexity of $\partial B_r(p)$, there exists a nested map $E:[0, 1] \rightarrow \cC(M)$ with $E(0) = \Omega_1, E(1) = M_+$ and $\Omega_1 \Delta M_+ \subset M_+\cap B_r(p)$ satisfying
            \begin{itemize}
                \item $\partial \Omega_1 \cap B_r(p)$ is minimal and strictly one-sided area minimizing in $M_+ \setminus \Omega_1$ (\cite[Lemma~15]{chodoshSingularBehaviorGeneric2020});
                \item $\mathbf{M}(\partial E(x)) \leq \mathrm{Area}(\Sigma)$;
                \item $\mathbf{M}(\partial \Omega_1) < \mathrm{Area}(\Sigma) - 2\varepsilon_0$, where $\varepsilon = \varepsilon_0(r) > 0$.
            \end{itemize}
            
            We can construct an intermediate nested sweepout $\{\Phi'(x)\}_{x \in [0, x_0]}$ by concatenating $\set{\Phi(x)\cap \Omega_1}_{x\in[0, x_0]}$ and $E(x)$, up to reparametrization. By the first bullet above, we have for $x \in [0, x_0]$,
            \[
                \mathbf{M}(\partial(\Phi(x) \cap \Omega_1)) \leq \mathbf{M}(\partial \Phi(x))\,.
            \]
            So together with the second bullet and a similar construction for $\{\Phi'(x)\}_{x \in [x_0, 1]}$ on $M_-$, $\{\Phi'(x)\}_{x \in [0, 1]}$ is still an optimal nested sweepout. 

            Then, let $\Sigma_1 = \partial^* \Omega_1$. Let $\tau > 0$ small enough depending on $\varepsilon_0$ and $U_0 = B_\tau(\Sigma_1)\cap \Omega_1$. By compactness of Caccioppoli sets, We can find a perimeter minimizer $\Omega_2$ with the constraint that $\Omega_1 - U_0 \subset \Omega_2 \subset \Omega_1$. 

            \begin{claim}
                $\partial^t \Omega_2 \cap \partial^t \Omega_1 = \emptyset$.
            \end{claim}
            \begin{proof}[Proof of laim 1.]
                As mentioned in Remark \ref{rem:red_topo_bdry}, we always assume that $\partial^t \Omega_1 = \overline{\partial^* \Omega_1}$ and $\partial^t \Omega_2 = \overline{\partial^* \Omega_2}$.

                We first note that $\Omega_2$ is also a perimeter minimizer with the constraint that $\Omega_1 - U_0 \subset \Omega_2 \subset M_+$. Indeed, if this is not true, we can find a perimter minimizer $\Omega'_2$ with $\Omega_1 - U_0 \subset \Omega'_2 \subset M_+$ such that
                \begin{equation}
                    \cH^{8}(\Omega'_2 \cap (M_+\setminus \Omega_1)) > 0\,.
                \end{equation}
                In other words, $\Omega'_2 \neq (\Omega'_2 \cap \Omega_1)$. However, by the first bullet above that $\partial \Omega_1$ is strictly one-sided area-minimizing in $M_+\setminus \Omega_1$, one can conclude
                \begin{equation}
                    P(\Omega'_2 \cap \Omega_1) < P(\Omega'_2)\,, 
                \end{equation}
                giving a contradiction.
                
                As $\Omega_2$ lies on one side of $M_+$, T. Ilmanen's strong maximum principle \cite{ilmanenStrongMaximumPrinciple1996} and Solomon-White strong maximum principle \cite{solomonStrongMaximumPrinciple1989} together imply that ${\partial^t\Omega_2\cap\overline\Sigma=\emptyset}$. 

                It suffices to verify that $\partial^t \Omega_2 \cap (\partial^t \Omega_1 \cap B_r(p)) = \emptyset$. Indeed, if this is not true, these maximum principles again imply
                \[
                    \partial^t \Omega_1\cap B_r(p)\subset \partial^t \Omega_2\,,
                \] so $\emptyset \neq \partial B_r(p) \cap \Sigma \subset \partial^t \Omega_2 \cap \Sigma$ contradicting to $\partial^t\Omega_2\cap\overline\Sigma=\emptyset$. 
            \end{proof}

            By taking $\tau > 0$ small enough, the interpolation lemma (Lemma \ref{lem:interpolation}) induces a nested map $E':[0, 1] \rightarrow \cC(M)$ with $E'(0) = \Omega_2, E'(1) = \Omega_1$ and $\mathrm{M}(\partial E'(x)) < \mathrm{Area}(\Sigma) - \varepsilon_0$, for any $x \in [0, 1]$.

            The desired sweepout $\set{\Psi(x)}_{x\in [0, x_0]}$ on $M_+$ is the reparameterized concatenation of $\set{\Phi'(x)\cap \Omega_2}_{x \in [0, x_0]}$, $\set{E'(x)}_{x \in [0,1]}$ and $\set{E(x)}_{x \in [0,1]}$. The conclusion holds on $M_+$, if $U \cap M_+ := M_+ \setminus \overline{\Omega_2}$ and $K \cap M_+ = M_+ \setminus \mathrm{int}(\Omega_1)$ with $r$ and $\varepsilon_0$ chosen above.

            The similar process can be done on $M_-$, and the new (ONVP) sweepout by $\set{\Psi(x)}_{x \in [0,1]}$, which satisfies the property. 
        \end{proof}
        \begin{Rem}\label{rem:tech}
            Due to Frankel property and monotonicity formula for minimal hypersurfaces, the open subset $U \setminus K$ can be chosen such that every minimal hypersurface $\Sigma'$ other than $\Sigma$ intersects $U \setminus K$. 
        \end{Rem}

        With the modification above, we can perturb the Riemmanian metric to obtain the unique realization property of min-max width.

        \begin{Lem}[Unique realization of min-max width] \label{lem:dichotomy}
            Given a closed Riemannian manifold $M^8$ with $H_7(M, \mathbb{Z}_2) = 0$, let $g$ be a metric in 
            \[
                \mathrm{int}(\scG_F) = \mathrm{int}\set{g \in \scG| (M, g) \text{ has Frankel property}}\,.
            \]
            Let $\Sigma$ be a minimal hypersurface realizing the min-max width $\cW$ generated from an (ONVP) sweepout as in Theorem \ref{thm:sing_index}. If $\mathfrak{h}_{nm}(\Sigma) =\set{p}$, then $\forall \varepsilon > 0$, there exists a metric $g' \in \mathrm{int}(\scG_F)$ with $\|g - g'\|_{C^k} < \varepsilon$ satisfying that $\cW(M, g')$ is uniquely realized by $\Sigma$ if generated from an (ONVP) sweepout. Furthermore, $\Sigma$ can be taken nondegenerate in $(M, g')$.
        \end{Lem}
        \begin{proof}
            By Lemma \ref{lem:tech}, we obtain a new sweepout $\Psi$ for $\Sigma$, a positive constant $\varepsilon_0 > 0$, and an open subset $\tilde U = U\setminus (K\cup \overline{\Sigma})$ therein.

            Firstly, let's choose a smooth function $f_1 \in C^\infty(M)$ such that $f_1$ is positive in $\tilde U$, vanishing outside $\tilde U$ and
            \begin{equation}\label{eqn:small}
                 (1+f_1)^7 \mathrm{Area}_g(\Sigma)\leq \mathrm{Area}_g(\Sigma) + \varepsilon_0 / 2\,.
            \end{equation}
            Note that if $\varepsilon_1\in (0, 1)$ small enough, $g_1 = (1 + \varepsilon_1 f_1)^2 g$ is still inside $\mathrm{int}(\scG_F)$. 

            Moreover, $\Psi$ is also an (ONVP) sweepout on $(M, g')$. Indeed, the metric only changes in $\tilde U$, so any slice intersecting $\tilde U$ will now has mass no greater than
            \begin{equation}
                (1 + \varepsilon_1 f_1)^7 (\mathrm{Area}_g(\Sigma) - \varepsilon_0) \leq \mathrm{Area}_g(\Sigma) - \varepsilon_0 / 2\,.
            \end{equation}
            Hence, $\cW(M, g_1) \leq \cW(M, g)$. Because $g_1 \geq g$, by definition, $\cW(M, g_1) \geq \cW(M, g)$ and we can conclude that $\cW(M, g_1) = \cW(M, g)$.
            
            Then, Let $q \in \Sigma$ and $r > 0$ small enough such that $B_r(q) \cap \overline\Sigma$ is regular and $B_r(q) \subset U$. We can choose another nonnegative smooth function $f_2 \in C^\infty_c(B_r(q))$ such that \[f_2(x)=\mathrm{dist}(x, \Sigma)^2 \eta(\mathrm{dist}(x, q))\,,\] where $\eta$ is a standard cut-off function. It is not hard to check that for $\varepsilon_2 > 0$ small enough, $g_2 = \mathrm{exp}(\varepsilon_2 f_2) g_1 \in \mathrm{int}(\scG_F)$ and similarly, $\cW(M, g_2) = \cW(M, g)$. Furthermore, the same proof of \cite[Lemma~4]{marquesEquidistributionMinimalHypersurfaces2019} together with the choice of $f_2$ implies that $\Sigma$ is nondegenerate in $(M, g_2)$.

            Finally, let $\Psi'$ be another (ONVP) sweepout in $(M, g_2)$ with $y_i \nearrow y_0 \in \mathbf{m}_L(\Psi')$, such that
            \begin{equation}
                |\partial \Psi'(y_i)| \rightarrow \kappa'|\Sigma'| \in \cR, \quad \kappa' \in \set{1,2}\,.
            \end{equation}
            It suffices to show that $\Sigma' = \Sigma$, and thus $\kappa' = 1$.

            Suppose by contradiction that $\Sigma' \neq \Sigma$; Also note that $\Sigma' \cap \Sigma \neq \emptyset$ since $g'\in \mathrm{int}(\scG_F)$ and thus, $\Sigma$ does not lie on either side of $\Sigma'$ by strong maximum principles \cite{ilmanenStrongMaximumPrinciple1996,solomonStrongMaximumPrinciple1989}. Let's take $\Psi'$ back to the original metric $g$ and obviously, $\Psi'$ is still an optimal nested sweepout, albeit $\kappa'|\Sigma'|$ may not be inside the critical set. Theorem \ref{thm:sing_index} implies that there exists a sequence $y'_i \rightarrow y'_0$ with $\kappa''|\Sigma''| = \lim_i |\partial \Psi(y'_i)| \in \cR$ realizing the min-max width with $\kappa'' \in \set{1, 2}$. By the nested property, we know that $\Sigma''$ should lie on one side of $\Sigma'$ or coincide with $\Sigma'$. Since $\Sigma' \neq \Sigma$ and $\Sigma$ does not lie on either side of $\Sigma'$, we see that $\Sigma'' \neq \Sigma$. By Frankel property of $(M, g)$, $\Sigma'' \cap \Sigma \neq \emptyset$. As mentioned in Remark \ref{rem:tech}, $\Sigma'' \cap \tilde U \neq \emptyset$ and thus,
            \begin{equation}
                \|\Sigma''\|(M, g_2) > \|\Sigma''\|(M, g) = \cW(M, g) = \cW(M, g_2)\,,
            \end{equation}
            contradicting the optimality of $\Psi'$.

            In summary, $g' = g_2$ is the desired metric.
        \end{proof}
    
    \subsection{Manifolds without Frankel Property}
        If $M$ doesn't have Frankel property, then there exist two minimal hypersurfaces $\Sigma_1$ and $\Sigma_2$ such that $\Sigma_1 \cap \Sigma_2 = \emptyset$. The topological assumption $H_7(M^8, \mathbb{Z}_2) = 0$ implies that $\Sigma_1$ and $\Sigma_2$ are two-sided and each of them separates the ambient manifold.

        According to Ilmanen's strong maximum principle \cite{ilmanenStrongMaximumPrinciple1996}, we can further obtain that
        \begin{equation}
            \overline{\Sigma}_1 \cap \overline{\Sigma}_2 = \emptyset\,.
        \end{equation}
        
        The goal of this subsection is to prove the following result.
        
        \begin{Prop}
            Given a closed Riemannian manifold $(M^{n+1}, g)$ without Frankel property but with $H_n(M, \mathbb{Z}_2) = 0$, there exists a locally one-sided area-minimizing hypersurface $\Sigma$ in $M$.
        \end{Prop}
        \begin{proof}
            Take $\Sigma_1$ and $\Sigma_2$ as above. If either of them is locally one-sided area-minimizing, then we are done.

            Suppose the neither of them is locally one-sided area-minimizing, and due to the separateness, there exists a connected component $N$ in $M \backslash(\overline{\Sigma}_1 \cup \overline{\Sigma}_2)$, such that $\partial N = \overline{\Sigma}_1 \cup \overline{\Sigma}_2$. Let $\Sigma$ be an area minimizer in $[\Sigma_1] \in H_7(\overline{N}, \mathbb{Z}_2)$. We claim that $\overline\Sigma \subset \mathrm{int}(N)$.

            On the one hand, $\Sigma$ can not be either $\Sigma_1$ or $\Sigma_2$ since they are not one-sided homologically area-minimizing.

            On the other hand, if $\overline \Sigma$ touches $\overline{\Sigma}_1$ or $\overline{\Sigma}_2$, then the intersection set should not be entirely inside $\mathrm{Sing}(\Sigma_1) \cup \mathrm{Sing}(\Sigma_2)$ due to Ilmanen's strong maximum principle again. However, if $\overline{\Sigma} \cap \Sigma_i \neq \emptyset$, Solomon-White's strong maximum principle \cite{solomonStrongMaximumPrinciple1989} implies that $\overline\Sigma_i \subset \overline\Sigma$, contradicting the area-minimizing property of $\Sigma$.

            In summary, $\overline\Sigma \subset \mathrm{int}(N)$ and is homologically area-minimizing in $\overline{N}$. Therefore, $\Sigma$ is also locally area-minimizing in $M$.
        \end{proof}

 \section{Singular Capacity of Minimal Hypersurfaces} \label{sect:sing_cap}
    Let $\scC$ be the space of stable minimal hypercones in $\RR^8$. By the standard dimension reduction argument, every cone in $\scC$ has a smooth cross section with $S^7(1)$. 

    Let $\scM$ be the space of triples $(\Sigma; M, g)$, where $(M, g)$ is an open subset of a Riemannian manifold and $\Sigma$ is a minimal hypersurface in $(M, g)$ with finitely many singular points. The topology on $\scM$ is induced by $C^4_{loc}$ convergence in $g$ with fixed $M$ and multiplicity one varifold convergence in $\Sigma$.

    \begin{Def} \label{Def_Sing Cap}
        A map $\SCap : \scM \to \NN$ is called a \textbf{singular capacity}, if 
        \begin{enumerate}
            \item[(i)] For every nontrivial $C\in \scC$ and every open subset $U\subset \RR^8$ containing the origin, we have\[
                    1\leq \SCap (C; \RR^8, g_{Euc}) = \SCap(C; U, g_{Euc}) <+\infty;    \]
                where $g_{Euc}$ is the Euclidean metric. We abbreviate for simplicity $\SCap (C; \RR^8, g_{Euc})$ to $\SCap(C)$;
            \item[(ii)] For every $(\Sigma; M, g)\in \scM$, 
                \begin{align*}
                    \SCap(\Sigma; M, g) := \sum_{p\in \mathrm{Sing}(\Sigma)} \SCap(C_p)\,,
                \end{align*}
                where $C_p$ is the unique tangent cone of $\Sigma$ at $p$ (conventionally, $\SCap(\Sigma; M, g) := 0$ if $\Sigma$ is smooth);
            \item[(iii)] If $(\Sigma_j; M, g_j) \to (\Sigma; M, g)$ in $\scM$ and $\Sigma_j$ is stable in $(M, g_j)$, then for every open subset $U\subset \subset M$ with $\partial U \cap \Sing(\Sigma) = \emptyset$, \[
                \SCap(\Sigma; U, g) \geq \limsup_{j\to \infty} \SCap(\Sigma_j; U, g_j)\,.   \]
        \end{enumerate}
    \end{Def}
  
    The main goal of this section is to prove the existence of singular capacity on $\scM$.
    
    \begin{Thm} \label{Thm_Sing Cap_Main}
        There exists a singular capacity $\SCap$ on $\scM$ satisfying the following condition.

        For every $\Lambda\geq 1$, there exists $N(\Lambda)\in \NN$ such that 
        \[
            \SCap(C) \leq N(\Lambda)
        \]
        for every $C\in \scC$ with density at $0$ less than or equal to $\Lambda$.
    \end{Thm}

    We start with a quantitative cone rigidity lemma, inspired by Cheeger-Naber \cite[Theorem 7.3]{cheegerQuantitativeStratificationRegularity2013}. Let
    \begin{align*}
        \theta(x, r; \mu) &:= \frac{\mu(B_r(x))}{r^7}\,,\\
        \theta(x; \mu) &:= \lim_{r \rightarrow 0_+}\theta(x, r; \mu)\,,\\
        \BB_r &:= \BB^8_r(0) \subset \RR^8\,,\\
        \scC_{\Lambda} &:=\{C\in \scC : \theta(0; \|C\|) \leq \Lambda\}\,.
    \end{align*}
    For simplicity, given two varifold $V_1$ on $(M, g)$ and $V_2$ on $(M, g_{Euc})$ with uniform volume bound, as long as $g$ and $g_{Euc}$ are close enough, we will view $V_2$ as a varifold on $(M, g_{Euc})$ and define $\mbfF(V_1, V_2)$ in $(M, g_{Euc})$.

    \begin{Lem} \label{Lem_Sing Cap_Cone rigidity}
    For any $\Lambda, \varepsilon>0$, there exists $\delta_1=\delta_1(\Lambda, \varepsilon)>0$ such that if $\Sigma$ is a stable minimal hypersurface in $(\BB_5, g)$ with $\|\Sigma\|(\BB_5)\leq \Lambda$ and $0 \in \overline\Sigma$ satisfying
    \begin{enumerate}
        \item[(i)] $\theta(0, 4; \|\Sigma\|) - \theta(0, 1; \|\Sigma\|) \leq \delta_1$;
        \item[(ii)] $\|g - g_\mathrm{Euc}\|_{C^4}\leq \delta_1$.
    \end{enumerate}

    Then there exists $C\in \scC$, $m\geq 1$ such that $\Sigma$ is $C^2$ $\varepsilon$-close to $m$ pieces of $C$ in $\AAa_{2,3}$ and $|\Sigma|$ is $\mbfF_{\BB_4}$ $\varepsilon$-close to $m|C|$.
    \end{Lem}
    \begin{proof}
        This is essentially a corollary of \cite{schoenRegularityStableMinimal1981}. Indeed, if this is false, we can find a sequence $\set{\Sigma_i}$ with $\delta_1(\Sigma_i) \leq \frac{1}{i}$ but $\varepsilon$ far away from any multiple pieces of any cone $C \in \scC$. By Schoen-Simon's compactness theorem \cite[Theorem~2]{schoenRegularityStableMinimal1981}, we know that $|\Sigma_i| \rightarrow V \in \cR$ in the varifold sense in $\BB_\frac{9}{2}$, and the support of V is a stable minimal hypersurface in $(\BB_\frac{9}{2}, g_\mathrm{Euc})$. Moreover, we have
        \begin{equation}
            \theta(0, 4; \|V\|) - \theta(0, 1; \|V\|) = 0\,,
        \end{equation}
        which implies that $C = \mathrm{supp}(V)$ is a stable minimal hypercone in $\mathbb{R}^8$. Thus, $C$ is smooth outside the origin.

        It follows immediately from \cite[Theorem~1]{schoenRegularityStableMinimal1981} that for $i$ large enough, $\Sigma_i$ is $C^2$ $\varepsilon$-close to $m$ pieces of $C$ in $\AAa_{2,3}$ and $|\Sigma|$ is $\mbfF_{\BB_4}$ $\varepsilon$-close to $m|C|$, contradicting to our assumption at the beginning of the proof.
    \end{proof}
  
    \begin{Lem} \label{Lem_Sing Cap_Cone only converges to multi 1 cone}
        For every $\Lambda>1$, there exists $\varepsilon(\Lambda)>0$ such that for any pair $C, C'\in \scC_\Lambda$ and every $m\geq 2$, \[
            \mbfF_{\BB_4}(|C'|, m|C|)\geq \varepsilon(\Lambda)\,.   \]   
    \end{Lem}  
    \begin{proof}
        Otherwise, for some $\Lambda>1$, there are $C_j, C_j' \in \scC_\Lambda$ and $m_j\geq 2$ such that
        \begin{equation}
            \mbfF_{\BB_4}(|C_j'|, m_j|C_j|)\to 0\,.
        \end{equation}
        By the monotonicity formula for minimal hypersurfaces, we have
        \[
            2\leq m:=\limsup_j m_j <\infty\,.
        \]

        By Schoen-Simon's compactness \cite{schoenRegularityStableMinimal1981} again, up to a subsequence, $|C_j|\to m'|C_\infty|$ for some $C_\infty \in \scC_\Lambda$ and $m'\geq 1$. Hence $|C_j'|$ subconverges to $mm'|C_\infty|$ multi-graphically near the cross section $C_\infty\cap \SSp^7$. By Sharp's compactness \cite{sharpCompactnessMinimalHypersurfaces2017}, $\{C_j\}$ induces a positive Jacobi field over $S_\infty:= C_\infty\cap \SSp^7 \subset \SSp^7$. This implies $S_\infty\subset \SSp^7$ is stable, which is impossible since $\SSp^7$ has positive Ricci curvature.
    \end{proof}
  
    \begin{Cor} \label{Cor_Sing Cap_Density Gap of subseq sing}
        For every $\varepsilon\in (0,1)$,  there exists $\delta_2 = \delta_2(\Lambda, \varepsilon) \in(0, 1)$ such that if $\Sigma \subset (\BB_5, g)$ is a stable minimal hypersurface with $\|g-g_{Euc}\|_{C^4}\leq \delta_2$ and $\mbfF_{\BB_5}(|\Sigma|, |C|)\leq \delta_2$ for some $C\in \scC_{\Lambda}$, then $\mathrm{Sing}(\Sigma)\cap \BB_4 \subset \BB_{\varepsilon}$. Moreover, For any $x \in \BB_1\cap \overline\Sigma$, we have
        \begin{itemize}
             \item either $\theta(x; \|\Sigma\|) \leq \theta(0; \|C\|) - 2\delta_2$;
             \item or $\Sing(\Sigma)\cap \BB_4 \subset \{x\}$.
        \end{itemize}
    \end{Cor}
    \begin{proof}
        For the first claim, suppose otherwise, there exist $\Lambda\geq 1$, $\varepsilon\in (0,1)$, a family of stable minimal hypersurfaces $\set{\Sigma_j \subset (\BB_5, g_j)}$, a family of stable minimal hypercones $\set{C_j}\subset\scC_{\Lambda}$  such that $\|g_j - g_{Euc}\|_{C^4} \to 0$, $\mbfF_{\BB_5}(|\Sigma_j|, |C_j|) \to 0$ but $\mathrm{Sing}(\Sigma_j)\cap \BB_4 \setminus \BB_\epsilon \neq \emptyset$.

        By Lemma \ref{Lem_Sing Cap_Cone only converges to multi 1 cone}, $|C_j| \rightarrow |C_\infty| \in \scC$. Hence, \[
            \mbfF_{\BB_5}(|\Sigma_j|, |C_\infty|) \leq \mbfF_{\BB_5}(|\Sigma_j|, |C_j|)+ \mbfF_{\BB_5}(|C_j|, |C_{\infty}|) \to 0\,.   \]
        \cite[Theorem~1]{schoenRegularityStableMinimal1981} implies that for $j$ sufficiently large, $\mathrm{Sing}(\Sigma_j) \subset \BB_\varepsilon$ contradicting to our assumption.

        For the second claim, we also argue by contradiction that there exists a family of stable minimal hypersurfaces $\set{\Sigma_j \subset (\BB_5, g_j)}$ as above with $\set{x_j\in \overline\Sigma_j\cap \BB_1}$ such that $\limsup_{j \to \infty}\theta(x_j; \|\Sigma_j\|) - \theta(0; \|C_j\|) \geq 0 $ and $x'_j \in \mathrm{Sing}(\Sigma_j)\setminus \{x_j\} \neq \emptyset$. Let $C_\infty$ be the same limit cone as above.

        If $C_\infty$ is a hyperplane, then by Allard regularity theorem \cite{allardFirstVariationVarifold1972}, we have $\mathrm{Sing}(\Sigma_j)\cap \BB_4 = \emptyset$ for $j>>1$, which violates our assumption.

        If $C_\infty$ is a nontrivial minimal cone, then by upper semi-continuity of density, Allard regularity and the fact that 
        \[
            \limsup_j\theta(x_j; \|\Sigma_j\|) \geq \theta(0; \|C_\infty\|) > 1\,,
        \]
        we have $x_j \to 0$ and $x'_j \to 0$. Hence, by monotonicity formula, Lemma \ref{Lem_Sing Cap_Cone rigidity} can be applied to $(\eta_{x_j, r_j})_{\sharp} (\Sigma_j)$, where $r_j = 2\mathrm{dist}(x_j, x'_j)/5$ to deduce that for sufficiently large $j$, $x'_j \notin \mathrm{Sing}(\Sigma_j)$, which also violates our assumption. 
    \end{proof}
  

  
    \begin{Lem} \label{Lem_Sing Cap_Uniform bd on |Sing|}
        For every $\Lambda>1$, there exists $N(\Lambda)\geq 1$ such that, for every $C\in \scC_\Lambda$ and any sequence of stable minimal hypersurfaces $\Sigma_j \subset (\BB_6, g_j)$ satisfying $(\Sigma_j; \BB_6, g_j)\to (C; \BB_6, g_{Euc})$ in $\scM$, we have \[
            \limsup_{j\to \infty} \sharp (\mathrm{Sing}(\Sigma_j)\cap \BB_4) \leq N(\Lambda)\,.   \]
    \end{Lem}
    \begin{proof}
        The lemma essentially follows from \cite{naberSingularStructureRegularity2020} in dimension $8$. For completeness, here we give a simpler and more self-contained proof.
   
        Let $\varepsilon_1 = \varepsilon(2\Lambda)$ given in Lemma \ref{Lem_Sing Cap_Cone only converges to multi 1 cone} and $\delta_3:= \min\{\delta_1(\varepsilon_1/10, 2\Lambda), \delta_2(\varepsilon_1/10, 2\Lambda)\}/10$, where $\delta_1$ is given by Lemma \ref{Lem_Sing Cap_Cone rigidity} and $\delta_2$ is given by Corollary \ref{Cor_Sing Cap_Density Gap of subseq sing}. 
   
        Clearly, it suffices to prove inductively that for each integer $0\leq k\leq 1+ (\Lambda-1)/\delta_3$, 
        \begin{align}
            \sup\{ \limsup_{j\to \infty} \sharp \left(\mathrm{Sing}(\Sigma_j)\cap \BB_4\right)  \} < +\infty   \label{Sing Cap_num of Sing uniformly bded}\,,  
        \end{align}    
        where the supremum is taken among all the sequences of stable minimal hypersurfaces $\{\Sigma_j \subset (\BB_6, g_j)\}_{j\geq 1}$ such that $(\Sigma_j; \BB_6, g_j)\to (C; \BB_6, g_{Euc})$ in $\scM$ for some $C\in \scC_{1+ k\delta_3}$.
   
        For $k = 0$, by volume monotonicity formula and Allard regularity, (\ref{Sing Cap_num of Sing uniformly bded}) holds and the upper bound could be taken to be $1$.

        Suppose (\ref{Sing Cap_num of Sing uniformly bded}) holds for $k-1$ but fails for $k$, and then there exists a family of stable minimal hypersurfaces $\Sigma_j \subset (\BB_6, g_j)$, with $(\Sigma_j; \BB_6, g_j) \to (C; \BB_6, g_{Euc})$ in $\scM$ for some $C\in \scC_{1+k\delta_3}$ but $\sharp (\Sing(\Sigma_j)\cap \BB_4) \to \infty$. It follows from Corollary \ref{Cor_Sing Cap_Density Gap of subseq sing} that the first bullet holds in $\BB_1 \cap \Sigma_j$ for $j >> 1$.

        Let $x_j\in \mathrm{Sing}(\Sigma_j) \subset \BB_1$, and we can define\[ 
            r_j:= \inf\{r>0: \theta(x_j, 4; \|\Sigma_j\|) - \theta(x_j, r; \|\Sigma_j\|)\leq 2\delta_3 \} > 0\,.   \]
        By Schoen-Simon compactness, $x_j \to 0 \in\RR^8$ and $r_j\to 0_+$. By the choice of $\delta_3$ and $r_j$ as well as the volume monotonicity formula, for $j >> 1$, \[
            \theta(x_j, 4s; \|\Sigma_j\|) - \theta(x_j, s; \|\Sigma_j\|) \leq \delta_1(\varepsilon_1/10, 2\Lambda)/2,\ \ \ \forall r_j< s\leq 1\,.  \]
        Hence by Lemma \ref{Lem_Sing Cap_Cone rigidity}, for $j>>1$ and each $s\in (r_j, 1]$, 
        \[
            \mbfF_{\BB_4}( (\eta_{x_j, s})_{\sharp} |\Sigma_j|, m^s_j |C^s_j| )\leq \varepsilon_1/10\,,
        \]
        for some $C^s_j\in \scC_{2\Lambda}$ and $m_j^s\in \NN$. Moreover, $\mathrm{Sing}(\Sigma_j) \subset \BB_{2r_j}$.

        On the one hand, for $s\in [1/2, 1]$, since $|\Sigma_j|\to |C|$, we have $m_j^{s} = 1$ and $C^s_j$ can be chosen to be $C$ for $j$ even larger. On the other hand, since for every pair of varifolds $V_1, V_2$ and every $r\in (0,1)$ we have \[
            \mbfF_{\BB_4}((\eta_{0,r})_{\sharp}V_1, (\eta_{0,r})_{\sharp}V_2) \leq \mbfF_{\BB_4}(V_1, V_2)/r\,,    \]
        thus we have 
        \begin{align*}
            \mbfF_{\BB_4}(m_j^s|C_j^s|, m_j^{2s}|C_j^{2s}|) &\leq \mbfF_{\BB_4}(m_j^s|C_j^s|, (\eta_{x_j,s})_{\sharp}|\Sigma_j|)+ \mbfF_{\BB_4}(m_j^{2s}|C_j^{2s}|, (\eta_{x_j,s})_{\sharp}|\Sigma_j|)\\ 
            &\leq \varepsilon_1 / 5 + 2\mbfF_{\BB_4}(m_j^{2s}|C_j^{2s}|, (\eta_{x_j,2s})_{\sharp}|\Sigma_j|)\\
            &< \varepsilon_1\,,
        \end{align*}
        where we utilise the dilation invariance of the cone $C_j^{2s}$. By the choice of $\varepsilon_1$ and lemma \ref{Lem_Sing Cap_Cone only converges to multi 1 cone}, we can conclude that $m_j^s = m_j^{2s}\equiv 1$ for $j>>1$ and $s\in (r_j, 1]$.

        By Lemma \ref{Lem_Sing Cap_Cone rigidity}, $\hat{\Sigma}_j:= (\eta_{x_j, r_j})_{\sharp} \Sigma_j$ subconverges to some stable minimal hypersurface $\Sigma_\infty \subset (\RR^8, g_{Euc})$. By Lemma \ref{Lem_Sing Cap_Cone only converges to multi 1 cone}, the limit varifold should have multiplicity one due to the fact that $m_j^{Kr_j} = 1$ for all $K>1$ and $j>>1$. Moreover, $\mathrm{Sing}(\Sigma_\infty)\subset \BB_3$ is a finite set containing $0$, and by volume monotonicity formula, 
        \begin{equation*}
            \begin{aligned}
                \theta(0, \infty; \|\Sigma_\infty\|) \leq \theta(0; \|C\|)\,, \\
                \theta(0, 1; \|\Sigma_\infty\|)\leq \theta(0; \|C\|) - 2\delta_3\,.
            \end{aligned}
        \end{equation*}
        Hence, by Corollary \ref{Cor_Sing Cap_Density Gap of subseq sing}, only the case in the first bullet occurs, i.e., for every $x'\in \mathrm{Sing}(\Sigma_\infty)$, $\theta(x'; \|\Sigma_\infty\|)\leq \theta(0; \|C\|) - 2\delta_3$.
   
        Since $\sharp (\mathrm{Sing}(\hat{\Sigma}_j)\cap \BB_4)\to \infty$ but $\Sigma_\infty$ only has finitely many singular points, there exists $\hat{x}\in\mathrm{Sing}(\Sigma_\infty)\cap\BB_4$ and $\rho_j\to 0_+$ such that $\sharp\big(\BB_{\rho_j}(\hat{x})\cap \mathrm{Sing}(\hat{\Sigma}_j)\big)\to \infty$. As the tangent cone of $\Sigma_\infty$ at $\hat{x}$ has density bounded above by $\theta(0; \|C\|)-2\delta_3 \leq 1+ (k-1)\delta_3$, the blow-up picture contradicts to the inductive assumption.

        By induction, we conclude the existence of such finite $N(\Lambda)$.
    \end{proof}

    \begin{proof}[Proof of theorem \ref{Thm_Sing Cap_Main}.]
        By (i) and (ii) of definition \ref{Def_Sing Cap}, it suffices to define $\SCap$ restricted to $\scC$ and then verify (iii). Let $\{\Lambda_k\}_{k\geq 0}$ be an increasing family of real numbers given by
        \begin{enumerate}
            \item[•] $\Lambda_0 := 1$;
            \item[•] $\Lambda_k := \Lambda_{k-1}+ \delta_2(2 +\Lambda_{k-1}, 1)$, where $\delta_2$ is given by Corollary \ref{Cor_Sing Cap_Density Gap of subseq sing} and can be assumed WLOG to be monotonically decreasing in $\Lambda$.
        \end{enumerate}

        For a trivial hyperplane $P \subset \mathbb{R}^8$, we have no choice but define $\SCap(P) := 0$.

        For a non-trivial $C\in \scC$ with $\theta(0; \|C\|)\in [\Lambda_{k-1}, \Lambda_k)$, let's define $\SCap(C):= \prod_{j=0}^{k}(1+N(\Lambda_j))$, where $N$ is given by lemma \ref{Lem_Sing Cap_Uniform bd on |Sing|}. 
   
        To verify Definition \ref{Def_Sing Cap} (iii), by \cite[Theorem~1]{schoenRegularityStableMinimal1981} again, it suffices to show that if $\Sigma_j \subset (\BB_6, g_j)$ are stable minimal hypersurfaces, $C\in \scC$ and $(\Sigma_j; \BB_6, g_j)\to (C; \BB_6, g_{Euc})$ in $\scM$, then \[
            \limsup_{j}\SCap(\Sigma_j; \BB_4, g_j) \leq \SCap(C)\,.\] 

        Suppose $\theta(0; \|C\|)\in [\Lambda_{k-1}, \Lambda_k)$, and we have the following three cases (up to subsequences).\\

        \paragraph*{\textbf{Case 1}} If $\sharp (\mathrm{Sing}(\Sigma_j)\cap \BB_4) >1$ for $j>>1$, then by Corollary \ref{Cor_Sing Cap_Density Gap of subseq sing}, each singularity $p$ of $\Sigma_j$ has density bounded above by $\theta(0; \|C\|)- 2\delta_2(2+\Lambda_{k-1}, 1)< \Lambda_{k-1}$ and thus,\[
            \SCap(C_p) \leq \prod^{k - 1}_{j = 0} (1 + N(\Lambda_j))\,.\]
        Therefore, \[
            \limsup_{j\to \infty} \SCap(\Sigma_j; \BB_4, g_j) \leq \prod_{j=0}^{k-1}(1+ N(\Lambda_j))\cdot \limsup_{j\to \infty} \sharp (\mathrm{Sing}(\Sigma_j)\cap \BB_4) \leq \SCap(C)\,.     \]\\  
   
        \paragraph*{\textbf{Case 2}} If $\mathrm{Sing}(\Sigma_j)\cap \BB_4$ is a single point $x_j$ for $j>>1$, then by volume monotonicity formula, $\limsup_j \theta(x_j; \|\Sigma_j\|) \leq \theta(0; \|C\|)<\Lambda_k$. Hence by definition, for $j >> 1$, $\theta(x_j; \|\Sigma_j\|) < \Lambda_k$ and \[
            \limsup_{j}\SCap(\Sigma_j; \BB_4, g_j) = \limsup_j \SCap(C_{x_j}) \leq \SCap(C)\,.\]\\

        \paragraph*{\textbf{Case 3}} If $Sing(\Sigma_j)\cap \BB_4 = \emptyset$ for $j>>1$, then apparently,
            \[
                \limsup_{j}\SCap(\Sigma_j; \BB_4, g_j) = 0 \leq \SCap(C)\,.
            \]
    \end{proof}\

    We will end this section with the following application of Singular Capacity.

    \begin{Lem} \label{Lem_Pf Main_General Perturb Lem}
        Let $\scG$ be the set of all $C^k$ ($k \geq 4$ or $k = \infty$) metrics on a closed Riemannian manifold $M^8$ and $\tilde{\scM}$ be a subspace of $\scM$, consisting of triples $(\Sigma; M, g)$ satisfying the following 
        \begin{enumerate}
            \item[(1)] $\forall (\Sigma; M, g)\in \tilde{\scM}$, $\Sigma$ is nondegenerate.
            \item[(2)] For any sequence $\set{(\Sigma_j; M, g_j)\in \tilde{\scM}}^\infty_{j = 1}$ and $(\Sigma_\infty; M, g_\infty) \in \tilde{\scM}$, with $g_j\to g_\infty$ in $C^k(M)$, we have \[
                    \liminf_j \mathrm{index}(\Sigma_j) = \mathrm{index}(\Sigma_\infty)\,,\] 
                and 
                \[
                    |\Sigma_j|\to |\Sigma_\infty|\,,
                \]
                in the varifold sense.
            \item[(3)] The projection map $\Pi: \tilde{\scM}\to \scG$ onto the third variable is injective.
        \end{enumerate}
        Then for every metric $g$ in the interior of $\overline{\Pi(\tilde{\scM})}\subset \scG$, there is a family of triples $(\Sigma_i; M, g_i)\in\tilde{\scM}$ such that $g_i\to g$ in $C^k$ and $\mathrm{Sing}(\Sigma_i) = \emptyset$.
    \end{Lem}
    \begin{proof}
        Let $g$ be a metric in the interior of $\overline{\Pi(\tilde{\scM})}$ and $\scU$ is an arbitrary $C^k$ neighborhood of $g$ which is also contained in $\overline{\Pi(\tilde{\scM})}$. Let $(\Sigma_0; M, g_0)\in\tilde{\scM}$ such that $g_0\in \scU$. Since each singular point of $\Sigma_0$ is isolated, we have $\SCap (\Sigma_0; M, g_0)< +\infty$.
     
        We shall prove inductively that there exists a sequence $\{(\Sigma_l; M, g_l)\}_{l\in \NN}\subset \tilde{\scM} \cap \Pi^{-1}(\scU)$ such that for each $l$,
        \begin{itemize}
            \item either $\mathrm{Sing}(\Sigma_l)=\emptyset$ (and take $(\Sigma_{l+1}; g_{l+1}):=(\Sigma_l; g_l)$);
            \item or $\SCap(\Sigma_{l+1}; M, g_{l+1})\leq \SCap(\Sigma_l; M, g_l) - 1$.
        \end{itemize}
        Note that by definition $\ref{Def_Sing Cap}$, $\SCap(\Sigma; M, g) = 0 \iff \mathrm{Sing}(\Sigma)=\emptyset$. Hence for sufficiently large $N> \SCap(\Sigma_0;M, g_0)$, $\Sigma_N$ is a closed smooth minimal hypersurface in $g_N$ and $(\Sigma_N; M, g_N)$ can be one item in the desired sequence.
     
        Suppose $(\Sigma_l, M, g_l)$ for some $l \in \NN$ has been constructed, and WLOG $\mathrm{Sing}(\Sigma_l) \neq \emptyset$. Let $\scF\subset~C_c^\infty(M\setminus\mathrm{Sing}(\Sigma_l))$ depending on $\Sigma_l, M$ and $g_l$ be specified as in Lemma \ref{Lem_Pre_Perturb to find MH reg} and fix an $f\in \scF$. Since $\Pi(\tilde{\scM})$ is $C^k$-dense in $\scU$, there exists a sequence of $(\Sigma^{(j)}; M, g^{(j)})\in \tilde{\scM}$ so that $g^{(j)}= g_l(1 + f^{(j)}/j)$ for some smooth functions $f^{(j)}\to f$ in $C^k$. 

        The definition of $\tilde{\scM}$ implies that $|\Sigma^{(j)}|\to |\Sigma_l|$. Let $\{U_p \ni p\}_{p\in \mathrm{Sing}(\Sigma_l)}$ be a finite pairwise disjoint family of open subsets of $M$ such that 
        \begin{equation}
            \mathrm{index}(\Sigma_l \setminus \bigcup U_p) = \mathrm{index}(\Sigma_l)\,.    
        \end{equation}
        By Lemma \ref{Lem_Pre_Perturb to find MH reg}, after passing to a subsequence, there exists some $p^*\in \mathrm{Sing}(\Sigma_l)$ such that $\mathrm{Sing}(\Sigma^{(j)})\cap U_{p^*} = \emptyset$ for $j>>1$. In addition, Condition (2) and the choice of $U_p$ implies that for sufficiently large $j$, in each $U_p$, $\Sigma^{(j)}$ is stable. Hence by Definition \ref{Def_Sing Cap}, 
        \begin{align*}
            \limsup_{j\to \infty} \SCap(\Sigma^{(j)}; M, g^{(j)}) 
                & = \limsup_{j\to \infty} \sum_{p\in \mathrm{Sing}(\Sigma_l)}\SCap(\Sigma^{(j)}; U_p, g^{(j)}) \\
                & \leq  \sum_{p^* \neq p\in Sing(\Sigma_l)}\SCap(\Sigma_l; U_p, g_l) \\
                & \leq \SCap(\Sigma_l; M, g_l) - \SCap(\Sigma_l; U_{p^*}, g_l) \\
                & \leq \SCap(\Sigma_l; M, g_l) - 1\,.
        \end{align*}
        We can then choose $(\Sigma_{l+1}, g_{l+1}) := (\Sigma^{(j)}, g^{(j)})$ for a sufficiently large $j$ such that $g^{(j)}\in \scU$ and $\SCap(\Sigma^{(j)}; M, g^{(j)}) \leq \SCap(\Sigma_l; M, g_l) - 1$.
    \end{proof}

\section{Proof of Theorem \ref{Thm_Intro_Main}}

    Given a closed Riemannian manifold $M^8$ with $H_7(M, \mathbb{Z}_2) = 0$ and $k \geq 4$ ($k$ could be $\infty$), we define 
    \begin{align*}
        \scG &:= \set{C^k \text{ Riemannian metrics on } M}\,;\\
        \scG_F &:= \set{g \in \scG| (M, g) \text{ has Frankel property}}\,;\\
        \scG_{NF} &:= \scG\backslash \scG_F\,;\\
        \scR &:= \set{g \in \scG| (M, g) \text{ admits a nondegenerate smooth minimal hypersurface}}\,.
    \end{align*}

    By B. White's structure theorem of smooth minimal hypersurfaces \cite{whiteSpaceMinimalSubmanifolds1991,whiteBumpyMetricsTheorem2017}, $\scR$ is an open subset in $\scG$. It suffices to show that $\scR$ is dense in $\scG$.

    Observe that for any given $g \in \scG_{NF}$, the candidate minimal hypersurface $\Sigma$ generated in Section \ref{sec:gen} is at least locally one-sided area-minimizing. Therefore, by Lemma \ref{lem:surgery}, $\scR$ is dense in $\overline{\scG_{NF}}$.

    Now, let's focus on its complement $\mathrm{int}(\scG_F)$ and we shall show that $\scR$ is a dense in $\mathrm{int}(\scG_F)$.
    
    \begin{proof}[Proof of theorem \ref{Thm_Intro_Main}.]
        As mentioned above, it suffices to show that $\scR$ is dense in $\mathrm{int}(\scG_F)$. Let 
        \begin{align*}
            \scG_F^s:= \{g\in \mathrm{int}(\scG_F):& \exists!\ \text{minimal hypersurface }\Sigma \text{ generated from an (ONVP)}\\ 
                &\text{sweepout and realizing }\cW(M, g), \text{ and }\Sigma\text{ is non-degenerate stable}\}      
        \end{align*}
     
        By Lemma \ref{lem:dichotomy}, $\forall g \in \mathrm{int}(\scG_F) \setminus \overline{\scG_F^s}$, there exists a minimal hypersurface $\Sigma$ in $(M, g)$ with $\mathfrak{h}_{nm}(\Sigma)=\emptyset$. Therefore, by Lemma \ref{lem:surgery}, 
        \begin{equation}\label{rel1}
            \mathrm{int}(\scG_F) \setminus \mathrm{int}(\overline{\scG_F^s}) \subset \overline{\mathrm{int}(\scG_F)\setminus \overline{\scG_F^s} }\subset \overline \scR\,.
        \end{equation}
        
        On the other hand, it's easy to verify that the set
        \begin{align*}
            \tilde{\scM}:= \{(\Sigma; M, g): g\in \overline{\scG_F^s},&\ \Sigma \text{ is the nondegenerate stable minimal hypersurface }\\
            & \text{generated from an (ONVP) sweepout and realizing }\cW(M, g)\}      
        \end{align*}
        satisfies the assumptions in Lemma \ref{Lem_Pf Main_General Perturb Lem}. Hence,
        \begin{equation}\label{rel2}
            \mathrm{int}(\overline{\scG_F^s}) \subset \overline\scR\,.
        \end{equation}

        Combining (\ref{rel1}) and (\ref{rel2}), we have $\mathrm{int}(\scG_F) \subset \overline\scR$. This completes the proof of Theorem~\ref{Thm_Intro_Main}.
    \end{proof}

\bibliography{reference}
\end{document}